\documentclass[10pt]{amsart}
\usepackage{amssymb,color,graphicx}
\newcommand{\Email}[1]{{\sl E-mail address:\/} {\rm\textsf{#1}}}
\newtheorem{thm}{Theorem}
\newtheorem{cor}[thm]{Corollary}
\newtheorem{lem}[thm]{Lemma}
\newtheorem{prop}[thm]{Proposition}
\newenvironment{ThmMain}[1]{\par\smallskip\noindent{\bf Theorem #1.}\sl}{\par\smallskip\noindent}

\newcommand{\R}{{\mathbb R}}
\renewcommand{\S}{{\mathbb S^{d-1}}}
\newcommand{\N}{{\mathbb N}}
\newcommand{\ir}[1]{\int_{\R}{#1}\;ds}
\newcommand{\ird}[1]{\int_{\R^d}{#1}\;dx}
\newcommand{\nrm}[2]{\|{#1}\|_{L^{#2}(\R^d)}}
\newcommand{\icnd}[1]{\int_{\mathcal C}{#1}\;dy}
\newcommand{\be}[1]{\begin{equation}\label{#1}}
\newcommand{\ee}{\end{equation}}
\renewcommand{\(}{\left(}
\renewcommand{\)}{\right)}
\newcommand{\finprf}{\unskip\null\hfill$\;\square$\vskip 0.3cm}

\def\cprime{$'$}

\begin{document}
\title[A logarithmic Hardy inequality]{A logarithmic Hardy inequality}
\author[M. del Pino, J. Dolbeault, S. Filippas, and A. Tertikas]{Manuel del Pino, Jean Dolbeault, Stathis Filippas, and Achilles Tertikas}
\address{M. del Pino: Departamento de Ingenier\'{\i}a Matem\'atica and CMM, Universidad de Chile, Casilla 170 Correo~3, Santiago, Chile.
\Email{delpino@dim.uchile.cl}\newline\indent
J. Dolbeault: Ceremade, Universit\'e Paris-Dauphine, Place de Lattre de Tassigny, 75775 Paris Cedex~16, France.
\Email{dolbeaul@ceremade.dauphine.fr}\newline\indent
S. Filippas: Department of Applied Mathematics, University of Crete, Knossos Avenue, 714 09 Heraklion
\& Institute of Applied and Computational Mathematics,
FORTH, 71110 Heraklion, Crete, Greece.
\Email{filippas@tem.uoc.gr}
\newline\indent
A. Tertikas: Department of Mathematics, University of Crete, Knossos Avenue, 714 09 Heraklion
\& Institute of Applied and Computational Mathematics,
FORTH, 71110 Heraklion, Crete, Greece.
\Email{tertikas@math.uoc.gr}}
\date{\today}
\begin{abstract}
We prove a new inequality which improves on the classical Hardy inequality in the sense that a nonlinear integral quantity with super-quadratic growth, which is computed with respect to an inverse square weight, is controlled by the energy. This inequality differs from standard logarithmic Sobolev inequalities in the sense that the measure is neither Lebesgue's measure nor a probability measure. All terms are scale invariant. After an Emden-Fowler transformation, the inequality can be rewritten as an optimal inequality of logarithmic Sobolev type on the cylinder. Explicit expressions of the sharp constant, as well as minimizers, are established in the radial case. However, when no symmetry is imposed, the sharp constants are not achieved among radial functions, in some range of the parameters.
\end{abstract}
\maketitle
\thispagestyle{empty}

\section{Introduction and main results}\label{Sec:Setting}

The classical Hardy inequality in $\R^d$, $d\ge 3$, states that for any smooth, compactly supported function $u\in \mathcal D(\R^d)$, the following inequality holds:
\be{1}
\ird{\frac{|u|^2}{|x|^{2}}}\,\le\,\frac 4{(d-2)^2}\ird{|\nabla u|^2}\;.
\ee
The constant $4/{(d-2)^2}$ is the best possible one. Many studies have been devoted to extensions and improvements of Hardy's inequality in bounded domains containing zero. In this direction, the first result is due to Brezis and V\'{a}zquez; see \cite{BrV}. In \cite{MR1760280}, nonlinear improvements have been established, whereas in \cite{MR1918494,MR2462588,MR1862130} linear and Sobolev type improvements are given. In \cite{Adi-Filippas-Tertikas}, the best constant in the correction term of Sobolev type is computed. We also refer to \cite{Cianchi-Ferone} for improvements involving nonstandard correction terms. A recent trend seems to be oriented towards weights involving a distance to a manifold rather than a distance to a point singularity; see for instance \cite{MR1918928,MR2060479, MR2198838,Tertikas-Tintarev}. In particular, when taking distance to the boundary, the dependence of the correction term on the geometry of the domain has been established in \cite{MR1892180, MR2214621,MR2354734}. In the special case of the half-space in three space dimensions, the best constant of the Sobolev term in the improvement of Hardy's inequality has been found in \cite{MR2424899} and it turns out to be the best Sobolev constant.

On the other hand a subject of particular interest has been the analysis of the link between Hardy's inequality \eqref{1} and Sobolev's inequality. A family of inequalities that interpolate between Hardy and Sobolev inequalities is given by the {\em Hardy-Sobolev inequality},
\be{Ineq:HS}
\(\ird{\frac{|u|^p}{|x|^{d-\frac{d-2}2\,p}}}\)^\frac 2p\leq\mathsf C_{\rm HS}(p)\ird{|\nabla u|^2}
\ee
for any $u\in \mathcal D(\R^d)$, where $2\leq p\leq{2\,d}/(d-2)$, $d\ge3$, for a certain $\mathsf C_{\rm HS}(p)>0$. Extremals for \eqref{Ineq:HS} are radially symmetric and the best constant $\mathsf C_{\rm HS}(p)$ can be explicitly computed: see \cite{MR1223899,MR1731336,Catrina-Wang-01,0902}. We shall recover the expression of $\mathsf C_{\rm HS}(p)$ at the end of Section~\ref{Sec:inTerpRadial}. Extensions and improvements of the Hardy-Sobolev inequalities, and more generally of the Caffarelli-Kohn-Nirenberg inequalities established in \cite{Caffarelli-Kohn-Nirenberg-84}, have been the object of many papers. We refer the reader for instance to \cite{MR1918928,MR2003359,MR2142067,MR2198838,Tertikas-Tintarev,Adi-Filippas-Tertikas} for various contributions to this topic.

\medskip The purpose of this paper is to investigate the connection between \eqref{1} and another classical Sobolev type inequality: the optimal \emph{logarithmic Sobolev inequality} in $\R^d$ established in \cite{Gross75} which, expressed in a scale invariant form due to Weissler in \cite{MR479373}, reads
\be{21}
\ird{ |u|^2 \log {|u|^2}}\,\le \, \frac d2\, \log \left (\frac 2{\pi\,d\,e}\ird { |\nabla u|^2}\right)
\ee
for any $u\in H^1(\R^d)$ such that $\ird{u^2}=1$. We point out a parallel between these inequalities: just like~ \eqref{1} is an endpoint of the family \eqref{Ineq:HS}, that connects with Sobolev's inequality, Inequality \eqref{21} can be viewed as an endpoint of a family of optimal Gagliardo-Nirenberg inequalities that also connects to Sobolev's inequality; see \cite{MR1940370,MR1957678} for more details.

We emphasize that Hardy's inequality \eqref{1} in $\R^d$ cannot be improved in the usual sense, that is, there is no nontrivial potential $V \geq 0$ and no exponent $q>0$ such that, for any function $u$,
\[
\mathsf C\(\ird{V(x)\,|u|^q}\)^{2/q}\le\,\frac 4{(d-2)^2}\ird{|\nabla u|^2}-\ird{\frac{|u|^2}{|x|^{2}}}
\]
for some positive constant $\mathsf C$, as one can easily see by testing the above inequality with $u_{\epsilon}(x)=|x|^{-\frac{d-2}{2}+\epsilon}$, $|x|\leq 1$, and $u_{\epsilon}(x)=|x|^{-\frac{d-2}{2}-\epsilon}$, $|x|>1$, and sending $\epsilon$ to zero.

Instead of improving on the potential, we study here the possibility of improving on the control of $u$. The weight is fixed to be $1/|x|^2$ and we try to get a control on $|u|^2\log |u|^2$ instead of a control on $|u|^2$ only, as can sometimes be done for inequalities which appear as endpoints of a family, like \eqref{21}. As a result, we obtain inequalities of logarithmic Sobolev type, with weight $1/|x|^2$ in the term involving the logarithm. Such an inequality is somewhat unusual, because in most of the cases, logarithmic Sobolev inequalities involve bounded positive measures. The euclidean case with Lebesgue's measure is an exception and can actually be reinterpreted in terms of the gaussian measure, see for instance \cite{MR2079069,MR2201954} for some recent contributions in this direction. In the case of bounded measures, there is a huge literature: one can refer to \cite{MR0311856,MR1682772,ca-ba-ro2} for a few key contributions.

The logarithmic Sobolev and Hardy inequalities play an important role in a number of instances. The first one is a very natural tool for obtaining intermediate asymptotics for the heat equation, see \cite{MR772092,MR1447044,MR1842428,0528,MR2448650,bartier2009improved} with natural extensions to nonlinear diffusions (see for instance \cite{Blanchet:2009sf,bonforte-2009} and references therein). These inequalities are also useful in obtaining heat kernel estimates (see for instance \cite{Gross75,MR1103113}). A~related logarithmic Sobolev inequality recently appeared in \cite{MR2308757,FMT}, where it was used for obtaining upper bounds for the heat kernel of a degenerate equation.

\medskip We shall denote by $ \mathcal D^{1,2}(\R^d)$ the completion of $\mathcal D (\R^d)$ under the $L^2(\R^d)$ norm of the gradient of $u$. Let
\[
\mathsf S= \frac 1{\pi\,d\,(d-2)}\,\left[\frac{\Gamma\(d\)}{\Gamma\(\frac d2\)}\right]^\frac 2d=\mathsf C_{\rm HS}\(\frac{2\,d}{d-2}\)
\]
be the optimal constant in Sobolev's inequality, according to \cite{MR0448404,MR0463908}. Our first result states the validity of the following {\em logarithmic Hardy inequality}.\par\medskip\noindent
\begin{ThmMain}{A} Let $d\ge 3$. There exists a constant $\mathsf C_{\rm LH}\in (0, \mathsf S]$ such that, for all $u\in \mathcal D^{1,2}(\R^d)$ with $ \ird{ \frac{|u|^2}{|x|^2}}= 1$, we have
\be{Ineq:LogHardy}
\ird {\frac{|u|^2}{|x|^{2}}\,\log \(|x|^{d-2} |u|^2 \) } \leq \frac d2\,\log \left[ \mathsf C_{\rm LH}\,\ird{|\nabla u|^2}\right].
\ee
\end{ThmMain}
Inequality \eqref{Ineq:LogHardy} can be viewed as an infinitesimal form of the Hardy-Sobolev inequality at $p=2$: we observe that its left hand side is nothing but the derivative in $p$ at $p=2$ of the left hand side of \eqref{Ineq:HS}, up to a factor~$2$. Compared to an entropy term with respect to the measure $|x|^{-2}\,dx$, there is however a $\log(|x|^d)$ term. Such a term is easily recovered by scaling considerations and compensates for the presence of a superquadratic nonlinearity $|u|^2\log|u|^2$. The quantities involved in \eqref{Ineq:LogHardy} give a precise account of the fact that, to exert control by the Dirichlet integral of a power larger than two of $u$, the singularity has to be at the same time milder.

It is natural to search for the optimal constant and extremals for Inequality \eqref{Ineq:LogHardy}. Our second result answers this question in the class of radially symmetric functions, depending only on $|x|$, $x\in\R^d$.
\begin{ThmMain}{B} Let $d\ge 3$. If $u =u(|x|)\in \mathcal D^{1,2}(\R^d)$ is radially symmetric, and $ \ird{ \frac{|u|^2}{|x|^2}}= 1$, then
\[
\ird {\frac{|u|^2}{|x|^{2}}\,\log \(|x|^{d-2}|u|^2 \) } \leq \frac d2 \log \left[\mathsf C^*_{\rm LH}\,\ird{|\nabla u|^2}\right].
\]
where
\[
\mathsf C^*_{\rm LH}:=\frac 4d\,\frac{\left[\Gamma\(\frac d2\)\right]^\frac 2d}{\pi\,(8\,\pi\,e)^\frac 1d}\, \left[\frac{d-1}{(d-2)^2}\right]^{1-\frac 1d}.
\]
Equality in the above inequality is achieved by the function
$$ u= \frac {\tilde u}{\ird{ \frac{|\tilde u|^2}{|x|^2}}}\quad \hbox{\rm where}\quad \tilde u(x)= |x|^{-\frac{d-2}2}\,\exp\(-\tfrac{(d-2)^2}{4\,(d-1)}\big[\log |x|\,\big]^2\)\,.$$
\end{ThmMain}
For $d \ge 2$ and $a<(d-2)/2$, by starting from a more general weighted Hardy inequality,
\be{Ineq:HaGen}
\ird{\frac{|u|^2}{|x|^{2\,(a+1)}}}\leq \frac 4{(d-2-2\,a)^2}\ird{\frac{|\nabla u|^2}{|x|^{2\,a}}}\;,
\ee
we prove the validity of a whole class of \emph{weighted logarithmic Hardy inequalities}. If we denote by $ \mathcal D^{1,2}_{a}(\R^d)$ the completion with respect to the norm defined by the right hand side of \eqref{Ineq:HaGen} of $\mathcal D (\R^d\setminus\{0\})$ if $d\ge 2$ and of $\{u\in\mathcal D(\R)\,:\,u'(0)=0\}$ if $d=1$, our result reads:
\begin{ThmMain}{A'} Let $d\ge 1$. Suppose that $a<(d-2)/2$, $\gamma \ge d/4$ and $\gamma>1/2$ if $d=2$. Then there exists a positive constant $\mathsf C_{\rm GLH}$ such that, for any $u \in \mathcal D^{1,2}_{a}(\R^d)$ normalized by $ \ird{ \frac{|u|^2}{|x|^{2\,(a+1)}}}= 1$, we have
\be{Ineq:GLogHardy}
\ird{\frac{|u|^2}{|x|^{2\,(a+1)}}\,\log \(|x|^{d-2-2\,a}\,|u|^2 \)}\leq 2\,\gamma\,\log\left[\mathsf C_{\rm GLH}\, \ird{\frac{|\nabla u|^2}{|x|^{2\,a}}}\right]\,.
\ee
\end{ThmMain}
On the other hand, in the radial case, we have a more general family of sharp inequalities:
\begin{ThmMain}{B'} Let $d \ge 1$, $a <(d-2)/2$ and $\gamma\ge1/4$. If $u =u(|x|)\in \mathcal D^{1,2}_a(\R^d)$ is radially symmetric, and $ \ird{ \frac{|u|^2}{|x|^{2\,(a+1)}}}= 1$, then
\be{Ineq:GLogHardyRad}
\ird{\frac{|u|^2}{|x|^{2\,(a+1)}}\,\log \(|x|^{d-2-2\,a}\,|u|^2 \)}\leq 2\,\gamma\,\log\left[\mathsf C_{\rm GLH}^*\, \ird{\frac{|\nabla u|^2}{|x|^{2\,a}}}\right]\,,
\ee
where
\be{Cglh}
\mathsf C_{\rm GLH}^*= \frac{1}{\gamma}\, \frac{ \left[ \Gamma \(\frac{d}{2}\) \right]^{\frac{1}{2\,\gamma}}}{(8\,\pi^{d+1}\,e)^{\frac{1}{4\,\gamma}}} \(\frac{4\,\gamma -1}{(d-2-2\,a)^2}\)^{ \frac{4\,\gamma -1}{4\,\gamma}}\quad\mbox{if}\quad\gamma>\frac 14\quad\mbox{and}\quad\mathsf C_{\rm GLH}^*= 4\, \frac{ \left[ \Gamma \(\frac{d}{2}\) \right]^2}{8\,\pi^{d+1}\,e}\quad\mbox{if}\quad\gamma=\frac 14\;.
\ee
If $\gamma>\frac 14$, equality in \eqref{Ineq:GLogHardyRad} is achieved by the function
$$ u= \frac {\tilde u} {\ird{ \frac{|\tilde u|^2}{|x|^2}}}\quad \hbox{where}\quad \tilde u(x)= |x|^{-\frac{d-2-2\,a}2}\,\exp\(-\tfrac{(d-2-2\,a)^2}{4\,(4\,\gamma-1)}\big[\log |x|\,\big]^2\)\,.$$
\end{ThmMain}
Theorems A and B are special cases of Theorems A' and B' corresponding to $a=0$, $\gamma=d/4$, $d\ge 3$. The family of inequalities of Theorem B' imply on the one hand the logarithmic Sobolev inequality, and on the other hand the Hardy inequality, with optimal constants, as we shall see in Section~\ref{Sec:LogRadial}. In dimension $d=1$, radial symmetry simply means that functions are even.

We notice that Inequalities \eqref{Ineq:GLogHardy} and \eqref{Ineq:GLogHardyRad} are both homogeneous and scale invariant. Actually, all integrals are individually scale invariant, in the sense that their values are unchanged if we replace $u(x)$ by $u_ \lambda(x)=\lambda^{(d-2-2\,a)/2}\,u(\lambda\,x)$. This is of course consistent with the fact that the inequalities behave well under the Emden-Fowler transformation
\be{Transf:Emden-Fowler}
u(x)=|x|^{-\frac{d-2-2\,a}2}\,w(y)\quad\mbox{with}\quad y=(s,\omega):=\(-\log |x|,\frac x{|x|}\)\in\mathcal C := \R \times \S
\ee
and have an equivalent formulation on the cylinder $\mathcal C$, which goes as follows.
\begin{ThmMain}{A''} Let $d\ge 1$, $a <(d-2)/2$, $\gamma \ge d/4$ and $\gamma>1/2$ if $d=2$. Then, for any $w \in H^1(\mathcal C)$ normalized by $ \icnd{w^2}= 1$, we have
\be{Ineq:GLogHardy-w}
\icnd{|w|^2 \, \log |w|^2 }\leq 2\,\gamma\,\log\left(\mathsf C_{\rm GLH}\, \left[\icnd{|\nabla w|^2}+ \frac14\,(d-2-2\,a)^2\right] \right)\,.
\ee
\end{ThmMain}
The optimal constant $\mathsf C_{\rm GLH}$ is the same in Theorems A' and A''. Similarly, to the case of radial functions depending only on $|x|$ corresponds the case of functions depending only on $s=-\log |x|$.
\begin{ThmMain}{B''} Let $d\ge 1$, $a <(d-2)/2$ and $\gamma\ge1/4$. If $w \in H^1(\mathcal C)$ depends only on $s\in\R$ and is normalized by $ \icnd{w^2}= 1$, then
\be{Ineq:GLogHardyRad-w}
\icnd{|w|^2 \, \log |w|^2 }\leq 2\,\gamma\,\log\left(\mathsf C_{\rm GLH}^*\, \left[\icnd{|\nabla w|^2}+ \frac14\,(d-2-2\,a)^2\right] \right)\,.
\ee
The value of the optimal constant $\mathsf C_{\rm GLH}^*$ is given by \eqref{Cglh}. If $\gamma>\frac 14$, equality in \eqref{Ineq:GLogHardyRad-w} is achieved by the function
\[
w(s)= \frac{\tilde w(s)}{ \icnd{\tilde w^2}}\quad \hbox{\rm where}\quad \tilde w(s) = \exp \(-\frac{(d-2-2\,a)^2}{4\,(4\,\gamma-1)}s^2 \)\,.
\]
\end{ThmMain}
If $d=1$, $\mathcal C$ is equal to $\R$. For any $d\ge 1$, one may suspect that the optimal constant for $\eqref{Ineq:GLogHardy}$ (resp. \eqref{Ineq:GLogHardy-w}) is achieved in the class of radially symmetric functions (resp. functions depending only on $s\in\R$) and therefore $ \mathsf C_{\rm GLH}= \mathsf C^*_{\rm GLH}$. Using the method developed in \cite{Catrina-Wang-01,Felli-Schneider-03,DET}, it turns out that there is a range of the parameters $a$ and $\gamma$ for which this is not the case.
\begin{ThmMain}{C} Let $d \geq 2$ and $a<-1/2$. Assume that $\gamma>1/2$ if $d=2$. If, in addition,
\[
\frac{d}{4}\leq \gamma<\frac 14 + \frac{(d-2\,a-2)^2}{4\,(d-1)}\,,
\]
then the optimal constant $\mathsf C_{\rm GLH}$ in inequality $\eqref{Ineq:GLogHardy}$ is not achieved by a radial function and $\mathsf C_{\rm GLH}> \mathsf C^*_{\rm GLH}$.
\end{ThmMain}
This paper is organized as follows. In Section \ref{Sec:interpol}, we derive Theorems A, A' and A'' as a consequence of a Caffarelli-Kohn-Nirenberg interpolation inequality. In Section \ref{Sec:radial}, we present a complete study of the radial case and in particular we prove Theorems B, B' and B''. This study is based on a sharp one-dimensional interpolation inequality. In Section \ref{Sec:LogRadial}, we show that Theorem B' implies both the logarithmic Sobolev inequality \eqref{21} and the Hardy inequality \eqref{1}. In the final section, we study the symmetry breaking of the interpolation inequalities as well as of the logarithmic Hardy inequality, thus establishing Theorem~C.

\section{Interpolation inequalities. Proof of Theorems A, A' and A''}\label{Sec:interpol}

In this section, we will give the proofs of Theorems A, A' and A'' with the help of a general inequality of Caffarelli-Kohn-Nirenberg type and a differentiation procedure with respect to some of the parameters of the inequality. Our starting point is the following inequality, which has been established in \cite{Caffarelli-Kohn-Nirenberg-84}:
\be{Ineq:CKN}
\(\ird{\frac{|u|^p}{|x|^{b\,p}}}\)^\frac 2p\leq \mathsf C_{\rm CKN}(p,a) \ird{\frac{|\nabla u|^2}{|x|^{2\,a}}}\quad\forall\;u\in \mathcal D (\R^d)\;.
\ee
Restrictions on the exponents are given by the conditions: $b \in (a+1/2,a+1]$ in case $d=1$, $b \in (a,a+1]$ when $d=2$ and $b \in [a,a+1]$ when $d \ge 3$. In addition, for any $d \ge 1$, we assume that
\be{Range}
a<\frac{d-2}{2}\quad\mbox{and}\quad p=\frac{2\,d}{d-2+2\,(b-a)}\,.
\ee
See for instance \cite{Catrina-Wang-01} for a review of various known results like existence of optimal functions. In the limit case $b=a+1$, $p=2$, \eqref{Ineq:CKN} is equivalent to \eqref{Ineq:HaGen} and the optimal constant is then $\mathsf C_{\rm CKN}(2,a)=4/(d-2-2\,a)^2$.

The range $a>(d-2)/2$ can also be covered with functions in the space $\mathcal D(\R^d\setminus\{0\})$. Inequalities are not restricted to spaces of smooth functions and can be extended to the space $\mathcal D_{a,b}(\R^d)$ obtained by completion of $\mathcal D(\R^d\setminus\{0\})$ with respect to the norm defined~by
\[
\|u\|^2=\nrm{\,|x|^{-b}\,u\,}p^2+\nrm{\,|x|^{-a}\,\nabla u\,}2^2\,,
\]
but some care is required. For instance, if $a>(d-2)/2$, it turns out that, for any $u\in\mathcal D_{a,b}(\R^d)$,
\[
\lim_{r\to 0_+}r^{-d}\,\|u\|_{L^2(B(0,r))}=0\;.
\]
See \cite{DET} for more details.

A key issue for \eqref{Ineq:CKN} is to determine whether equality is achieved among radial solutions when $\mathsf C_{\rm CKN}$ is the optimal constant, or, alternatively, if symmetry breaking occurs. See \cite{Catrina-Wang-01,Felli-Schneider-03,MR2051129,DET,0902} for conditions for which the answer is known. Here are some cases for which radial symmetry holds:
\begin{enumerate}
\item[(i)] The dimension is $d=1$.
\item[(ii)] If $d\ge 3$, we assume either $a\ge 0$ or, for any $p\in(2,2^*)$, $a<0$ and $|a|$ is small enough, or for any $a<0$, $p-2>0$ is small enough.
\item[(iii)] If $d=2$, we assume either $a<0$ with $|a|$ small enough and $|a|\,p<2$, or, for any $a<0$, $p-2>0$ is small enough.
\end{enumerate}
In such cases, optimal functions are known and $\mathsf C_{\rm CKN}(p,a)$ is explicit (see Section~\ref{Sec:Radial-Interpolation}). Alternatively, it is known that for $d\ge 2$, if
\be{Cdt:FS}
a<b<1+a-\frac d2\(1-\frac{d-2-2\,a}{\sqrt{(d-2-2\,a)^2+4\,(d-1)}}\)\,,
\ee
minimizers are not radially symmetric. In such a case the explicit expression of $\mathsf C_{\rm CKN}$ is not known. More details will be given in Section~\ref{Sec:SymBreaking}.

Let $2^*=\infty$ if $d=1$ or $d=2$, $2^*=2d/(d-2)$ if $d\ge3$ and define
\[
\vartheta(p,d):=\frac{d\,(p-2)}{2\,p}\;.
\]
We have a slightly more general family of interpolation inequalities than \eqref{Ineq:CKN}, which has also been established in \cite{Caffarelli-Kohn-Nirenberg-84} and goes as follows.
\begin{thm}[According to \cite{Caffarelli-Kohn-Nirenberg-84}]\label{Thm:CKN} Let $d\ge 1$. For any $\theta\in[\vartheta(p,d),1]$, there exists a positive constant $\mathsf C(\theta,p,a)$ such that
\be{Ineq:Gen_interp}
\(\ird{\frac{|u|^p}{|x|^{b\,p}}}\)^\frac 2p\leq \mathsf C(\theta,p,a)\(\ird{\frac{|\nabla u|^2}{|x|^{2\,a}}}\)^\theta\(\ird{\frac{|u|^2}{|x|^{2\,(a+1)}}}\)^{1-\theta}\quad\forall\;u\in \mathcal D^{1,2}_{a}(\R^d)\;.
\ee
\end{thm}
Inequality \eqref{Ineq:Gen_interp} coincides with \eqref{Ineq:CKN} if $\theta=1$. We will establish the expression of $\mathsf C(\theta,p,a)$ when minimizers are radially symmetric and extend the symetry breaking results of Felli and Schneider to the case $\theta<1$ in Sections~\ref{Sec:radial} and \ref{Sec:SymBreaking} respectively. Before, we give an elementary proof of \eqref{Ineq:Gen_interp}, whose purpose is to give a bound on $\mathsf C(\theta,p,a)$ in terms of the best constant in \eqref{Ineq:CKN}, and to justify the limiting case that is obtained by passing to the limit $\theta\to 0_+$ and $p\to 2_+$ simultaneously. 
\begin{prop}\label{Thm:Extensionsb} Let $b \in (a+1/2,a+1]$ when $d=1$, $b \in (a,a+1]$ when $d=2$ and $b \in [a,a+1]$ when $d \ge 3$. In addition, for any $d \ge 1$, we assume that \eqref{Range} holds. Then we have\\
{\rm (i)} Let $\mathfrak K:=\{k\in(0,2)\,:\,k\le d-(d-2)\,p/2\;\mbox{if}\;d\ge3\}$. For any $\theta\in[\vartheta(p,d),1]\cap(1-2/p,1]$, we have
\[
\mathsf C(\theta,p,a)\le\inf_{k\in \mathfrak K}\left[\mathsf C_{\rm CKN}\(\frac{2\,(p-k)}{2-k},a\)\right]^{1-\frac kp}\,\(\frac 2{d-2-2\,a}\)^{2\(\frac kp+\theta-1\)}
\]
{\rm (ii)} Let $d\ge 2$. Suppose that $a<(d-2)/2$, $\gamma \ge d/4$ and $\gamma>1/2$ if $d=2$. We have
\[
\mathsf C_{\rm GLH}\le\mathsf C_{\rm CKN}\(\tfrac{4\,\gamma}{2\,\gamma-1},a\)\,.
\]\end{prop}
\begin{proof} If $d\ge 3$ and $p=2^*$, that is for $b=a$, then $\vartheta(p,d)=1=\theta$ and \eqref{Ineq:Gen_interp} is reduced to a special case of~\eqref{Ineq:CKN}. Assume that $p<2^*$. Let $u\in \mathcal D(\R^d)$. For any $k\in(0,2)$, we have:
\be{4.10}
\ird{\frac{|u|^p}{|x|^{b\,p}}}=\ird{\left(\frac{|u|}{|x|^{1+a}}\right)^{k}\left(\frac{|u|^{p-k}}{|x|^{b\,p-k\,(1+a)}}\right)}\leq\left(\ird{\frac{|u|^2}{|x|^{2\,(a+1)}}}\right)^{\frac{k}{2}}\left(\ird{ \frac{|u|^{2\,\frac{p-k}{2-k}}}{|x|^{2\,\frac{b\,p-k\,(1+a)}{2-k}}}}\right)^{\frac{2-k}{2}}.
\ee
We observe that $\eqref{Ineq:CKN}$ holds for some $a$, $b$ and $p$ if, due to the scaling invariance, these parameters are related by the relation
\[
b=a+1+d\,\(\tfrac 1p-\tfrac 12\)=a+1-\vartheta(p,1)\;.
\]
For any $k\in(0,2)$, we also have the relation $B=A+1+d\,\big(\frac 1P-\frac 12\big)$ if
\[
A=a\;,\quad P=\frac{2\,(p-k)}{2-k}\quad\mbox{and}\quad B\,P=\frac{2\,(b\,p-k\,(1+a))}{2-k}\;.
\]
Hence, using \eqref{Ineq:CKN}, we have that
\be{4.12}
\left(\ird{ \frac{|u|^{2\,\frac{p-k}{2-k}}}{|x|^{2\,\frac{b\,p-k\,(1+a)}{2-k}}}}\right)^{\frac{2-k}{p-k}}\leq \mathsf C_{\rm CKN}\(2\,\frac{p-k}{2-k},a\) \ird{ \frac{|\nabla u|^2}{|x|^{2\,a}}}
\ee
provided that $2<2\,(p-k)/(2-k)\leq 2^*$ if $d\ge 3$, which is equivalent to $k\leq d-(d-2)\,p/2$ using the fact that $k<2$. On the other hand, we may estimate the first integral in the right hand side of \eqref{4.10} by \eqref{Ineq:HaGen} and get
\be{4.14}
\ird{\frac{|u|^2}{|x|^{2\,(a+1)}}}\le \(\frac{4}{(d-2-2\,a)^2}\ird{\frac{|\nabla u|^2}{|x|^{2\,a}}}\)^{1-\alpha}\,\(\ird{\frac{|u|^2}{|x|^{2\,(a+1)}}}\)^{\alpha}
\ee
for any $\alpha \in [0,1]$. Combining \eqref{4.10}, \eqref{4.12} and \eqref{4.14} we get
\[
\(\ird{\frac{|u|^p}{|x|^{b\,p}}}\)^\frac 2p\leq \mathsf C(\theta,p,a)\(\ird{\frac{|\nabla u|^2}{|x|^{2\,a}}}\)^\theta\(\ird{\frac{|u|^2}{|x|^{2\,(a+1)}}}\)^{1-\theta}
\]
with $\theta = 1-\alpha k/p\in [\vartheta(p,d),1]$ and this proves \eqref{Ineq:Gen_interp}. Notice that for $d\ge 3$, the restriction $\theta\ge\vartheta(p,d)$ comes from the condition $k\leq d-(d-2)\,p/2$. If $d=2$, $\theta>\vartheta(p,2)=1-2/p$ is due to the restriction $k<2$. However, if $\theta=\vartheta(p,2)$, Inequality~\eqref{Ineq:Gen_interp} still holds true. For this case, we refer to \cite{Caffarelli-Kohn-Nirenberg-84}. If $d=1$, one knows that the Inequality~\eqref{Ineq:Gen_interp} holds true under the condition $\theta>1-2/p$, but the inequality still holds under the weaker condition $\theta\ge\vartheta(p,1)=1/2-1/p$. See Section~\ref{Sec:inTerpRadial} for further details in the one-dimensional case.

Let $P\in(2,2^*]$ if $d\ge 3$, $P>2$ if $d=1$ or $2$. For any $p\in [2, P)$ we choose $k =2\,\frac{P-p}{P-2}\in (0,2]$, which also satisfies $k\le d-(d-2)\,p/2$ so that $P=2\,\frac{p-k}{2-k}\leq 2^*$, if $d\ge 3$ and $k<2$. We also set $B:=a+1-d\,\(\frac 1P-\frac 12\)$, so that $B\,P=2\,\frac{b\,p-k\,(1+a)}{2-k}$. Then \eqref{4.10} can be written as
\be{4.16}
\ird{\frac{|u|^p}{|x|^{b\,p}}}\leq \left(\ird{\frac{|u|^2}{|x|^{2\,(a+1)}}}\right)^{\frac{P-p}{P-2}}\left(\ird{ \frac{|u|^{P}}{|x|^{BP}}}\right)^{\frac{p-2}{P-2}}.
\ee
Here we assume that $b=a+1+d\,\big(\tfrac 1p-\tfrac 12\big)$. If $d\ge 3$, Estimate \eqref{4.16} is valid for any $ 2 <p<P \leq 2^*$, and it is an equality for $p=2$ and any $P \in (2, 2^*]$. By differentiating \eqref{4.16} with respect to $p$ at $p=2$, we get that for any $P \in (2, 2^*]$,
\[
\ird{\frac{|u|^2}{|x|^{2\,(a+1)}}\log\(\frac{|x|^{d-2-2\,a}\,|u|^2}{\ird{\frac{|u|^2}{|x|^{2\,(a+1)}}}}\)}\leq \frac{P}{P-2}\ird{\frac{|u|^2}{|x|^{2\,(a+1)}}}\; \log\left[\frac{\(\ird{\frac{|u|^P}{|x|^{BP}}}\)^\frac{2}{P}}{\ird{\frac{|u|^2}{|x|^{2\,(a+1)}}}}\right]\,.
\]
For $d \ge 3$, let $2\,\gamma:=\frac{P}{P-2}\in [\frac{d}{2}, \infty)$. For $d=1$, $2$, let $2\,\gamma:=\frac{P}{P-2}\in (1,\infty)$. Then, for any $\gamma \ge d/4$ if $d\ge 3$ and any $\gamma>1/2$ if $d=1$, $2$ and any $a<(d-2)/2$, using once more \eqref{Ineq:CKN}, we have shown (ii).\end{proof}

\noindent {\em Proof of Theorems A, A' and A'':} The existence of $\mathsf C_{\rm GLH}$ is a straightforward consequence of Proposition~\ref{Thm:Extensionsb} if $d\ge2$. This proves Theorem A', except for the case $d=1$ which will be considered in Section~\ref{Sec:OneD}. Theorem A follows with $\gamma = d/4$, $d\ge3$ and $a=0$. In particular we get an upper bound for the optimal constant: $\mathsf C_{\rm LH}\leq \mathsf C_{\rm CKN}\(2^*, 0 \)=\mathsf S$.

By the Emden-Fowler transformation \eqref{Transf:Emden-Fowler}, the inequalities of Proposition~\ref{Thm:Extensionsb}, on $\R^d$, are transformed into equivalent ones on the cylinder $\mathcal C=\R \times \S$. More precisely \eqref{Ineq:Gen_interp} can be reformulated as
\be{Ineq:Gen_interpw}
\(\icnd{|w|^p}\)^\frac 2p\leq \mathsf C(\theta,p,a)\(\icnd{| \nabla w|^2}+\tfrac 14\,(d-2-2\,a)^2\icnd{|w|^2}\)^\theta\(\icnd{|w|^2}\)^{1-\theta}\,.
\ee
Notice that by standard arguments, the sharp constant in \eqref{Ineq:Gen_interpw} is achieved in $H^1(\mathcal C)$ when the parameters are in the range corresponding to the assumptions of Proposition \ref{Thm:Extensionsb}, provided $p>2$ and $\theta>\vartheta(p,d)$.

By \eqref{Transf:Emden-Fowler}, the logarithmic Hardy inequality \eqref{Ineq:GLogHardy} of Theorem A' takes the form \eqref{Ineq:GLogHardy-w} of Theorem A'' if $\gamma \geq d/4$, $d\ge3$ and $a<(d-2)/2$, or $\gamma >1/2$, $d=2$ and $a<0$. The one-dimensional case, for which $\mathcal C=\R$, will be directly investigated in the next section. 
\finprf

\section{The one-dimensional and the radial cases. Proof of Theorems B, B' and B''}\label{Sec:radial}

In this section we will study the Caffarelli-Kohn-Nirenberg interpolation inequality \eqref{Ineq:Gen_interpw} as well as the logarithmic Hardy inequality \eqref{Ineq:GLogHardy-w} in the one-dimensional case and under the restriction to the set of radial functions. As a consequence, we shall also establish Theorems B, B' and B''.

\subsection{The sharp interpolation inequality in the one-dimensional cylindric case}\label{Sec:inTerpRadial}

If $w$ depends only on $s=-\log |x|$, Inequality \eqref{Ineq:Gen_interpw} can be reduced to its one-dimensional version,
\be{Ineq:sigma}
\(\ir{|w|^p}\)^\frac 2p\leq \mathsf K(\theta,p,\sigma)\,\(\ir{|w'|^2}+ \sigma^2 \ir{|w|^2}\)^\theta\(\ir{|w|^2}\)^{1-\theta}\;,
\ee
for any $w\in H^1(\R)$, with $\sigma=(d-2-2\,a)/2$, provided $\mathsf C(\theta,p,a)\,|\S|^{1-2/p}=\mathsf K(\theta,p, \sigma)$. Inequality~\eqref {Ineq:sigma} is however of interest by itself and can be considered as depending on the parameters $\theta$, $p$ and $\sigma$, independently of $a$ and $d$.

If $\theta>\vartheta(p,1)$, the proof of the existence of an optimal function is standard. After optimizing the inequality under scalings, \eqref{Ineq:sigma} reduces to a Gagliardo-Nirenberg inequality, whose optimal function is defined up to a scaling and a multiplication by a constant. Let us give some details.

If we optimize Inequality~\eqref{Ineq:sigma} under scalings, we find that it is equivalent to the one-dimensional Gagliardo-Nirenberg inequality. Let
\[
Q[w]:=\(\ir{|w'|^2}+ \sigma^2 \ir{|w|^2}\)^\theta\(\ir{|w|^2}\)^{1-\theta}
\]
be the functional which appears in the right hand side of Inequality~\eqref{Ineq:sigma} and consider $w_\lambda(s)=\lambda^{1/p}\,w(\lambda\,s)$, $\lambda>0$. This scalings leaves the left hand side of Inequality~\eqref{Ineq:sigma} invariant, while
\[
Q[w_\lambda]=\(\lambda^{2-b}\ir{|w'|^2}+ \sigma^2\,\lambda^{-b}\ir{|w|^2}\)^\theta\(\ir{|w|^2}\)^{1-\theta}
\]
with $b=(p-2)/(p\,\theta)$. We observe that $a=2-b$ is positive if and only if $\theta>(p-2)/(2\,p)=\vartheta(p,1)$. Hence we find that
\begin{enumerate}
\item [(i)] if $\theta<\vartheta(p,1)$, then $\inf_{\lambda>0}Q[w_\lambda]=\lim_{\lambda\to\infty}Q[w_\lambda]=0$, and Inequality~\eqref{Ineq:sigma} does not hold.
\item [(ii)] if $\theta=\vartheta(p,1)$, then
\[
\inf_{\lambda>0}Q[w_\lambda]=\lim_{\lambda\to\infty}Q[w_\lambda]=\(\ir{|w'|^2}\)^{\vartheta(p,1)}\(\ir{|w|^2}\)^{1-{\vartheta(p,1)}}\,,
\]
so that \eqref{Ineq:sigma} is equivalent to the one-dimensional Gagliardo-Nirenberg inequality
\be{Ineq:GN1}
\|w\|_{L^p(\R)}\le\mathsf C_{\rm GN}\,\|\,w'\|_{L^2(\R)}^{\vartheta(p,1)}\,\|\,w\|_{L^2(\R)}^{1-{\vartheta(p,1)}}\quad\forall\;w\in H^1(\R)\;.
\ee
Hence $\mathsf K(\vartheta(p,1),p,\sigma)=\mathsf C_{\rm GN}^2$ is independent of $\sigma>0$, and Inequality~\eqref{Ineq:sigma} admits no optimal function if $\theta=\vartheta(p,1)$, $\sigma>0$. It degenerates into \eqref{Ineq:GN1} in the limit $\sigma\to0_+$, for which an optimal function exists.
\item [(iii)] if $\theta>\vartheta(p,1)$, then $\inf_{\lambda>0}Q[w_\lambda]$ is achieved for 
\[
\lambda^2=\frac ba\,\frac{\sigma^2\ir{|w|^2}}{\ir{|w'|^2}}\;,
\]
that is
\[
\inf_{\lambda>0}Q[w_\lambda]=\kappa\,\(\ir{|w'|^2}\)^{\vartheta(p,1)}\(\ir{|w|^2}\)^{1-{\vartheta(p,1)}}
\]
with $\kappa=\big(\frac a{b\,\sigma^2}\big)^\frac{p-2}{2\,p}\,\big(\frac{a+b}a\,\sigma^2\big)^\theta$, \emph{i.e.}
\[
\frac 1\kappa= \left[\frac{(p-2)\,\sigma^2}{2+(2\theta-1)\,p}\right]^\frac{p-2}{2\,p} \left[\frac{2+(2\theta-1)\,p}{2\,p\,\theta\,\sigma^2}\right]^\theta\,.
\]
As a consequence, $\mathsf K(\theta,p,\sigma)=\kappa^{-1}\,\mathsf C_{\rm GN}^2$ and optimality is achieved in Inequality~\eqref{Ineq:sigma}, since \eqref{Ineq:GN1} admits an optimal function; see for instance \cite{MR2427077,MR2317190}.
\end{enumerate}

The above Gagliardo-Nirenberg inequality \eqref{Ineq:GN1} is equivalent to the Sobolev inequality corresponding to the embedding $H^1_0(0,1)\hookrightarrow L^q(0,1)$ for some $q=q(p)>2$; see \cite{Benguria-Loss03,MR1923362,MR2048612} for more details. Also notice that, using the radial symmetry of the minimizers of the optimal functions of the Hardy-Sobolev inequality~\eqref{Ineq:HS}, we recover the expression of $\mathsf C_{\rm HS}(p)=|\S|^{-(p-2)/p}\,\mathsf K(1,p,(d-2)/2)$ that can be found~in~\cite{MR1223899,MR1731336,Catrina-Wang-01,0902}, using $|\S|=2\,\pi^{d/2}/\,\Gamma\(d/2\)$, $\sqrt\pi\,\Gamma(d)=2^{d-1}\,\Gamma\(d/2\)\,\Gamma\((d+1)/2\)$ and Lemma~\ref{Lem:interpRadial} below. Similarly, in all cases for which optimal functions are known to be radially symmetric in Inequality~\eqref{Ineq:CKN} (see Section~\ref {Sec:interpol}), we have $\mathsf C_{\rm CKN}(p,a)=|\S|^{-(p-2)/p}\,\mathsf K(1,p,(d-2-2\,a)/2)$. The constant $\mathsf K(\theta,p, \sigma)$ can be computed as follows.

\begin{lem}\label{Lem:interpRadial} Let $\sigma >0$, $p>2$ and $\theta \in \left[\vartheta(p,1), 1\right]$. Then the best constant in Inequality \eqref{Ineq:sigma} is given by:
\be{EqnK}
\mathsf K(\theta,p, \sigma)=\left[\frac{(p-2)^2\,\sigma^2}{2+(2\theta-1)\,p}\right]^\frac{p-2}{2\,p} \left[\frac{2+(2\theta-1)\,p}{2\,p\,\theta\,\sigma^2}\right]^\theta \left[\frac 4{p+2}\right]^\frac{6-p}{2\,p}\left[\frac{\Gamma\left(\frac{2}{p-2}+\frac 12\right)}{\sqrt\pi\;\Gamma\left(\frac{2}{p-2}\right)}\right]^\frac{p-2}{p}\,.
\ee
If $\theta>\vartheta(p,1)$, the best constant is achieved by an optimal function $\overline{w}(s)$, which is unique up to multiplication by constants and shifts and is given by
\[
\overline{w}(s)=\big(\cosh(\lambda\,s)\big)^{-\frac2{p-2}}\quad\mbox{with}\quad{\textstyle\lambda =\frac 12\,(p-2)\,\sigma\left[\frac{p+2}{2+(2\theta-1)\,p}\right]^{\frac12}}\,.
\]
\end{lem}
\begin{proof} Using the Emden-Fowler transformation, the value of $\mathsf K(\theta,p,\sigma)$ can be computed using the equation
\be{Eqn:ODE}
(p-2)^2\,w''-4\,w+2\,p\,|w|^{p-2}\,w=0
\ee
such that $w'(0)=0$ and $\lim_{|s|\to\infty}w(s)=0$. A minimizer for \eqref{Ineq:sigma} is indeed defined up to a translation (a scaling in the original variables) and a multiplication by a constant, which can be adjusted to fix one of the coefficients in the Euler-Lagrange equation as desired. An optimal function can therefore be written as $w(\lambda\,s)$ for some $\lambda>0$, on which we can optimize. The solution $w$ of \eqref{Eqn:ODE} is unique if we further assume that it is positive with a maximum at $s=0$. This can be seen as follows. Multiply \eqref{Eqn:ODE} by $w$ and integrate from $s$ to $+\infty$. Since $\lim_{|s|\to\infty}w'(s)=0$, the function $s\mapsto\frac 12\,(p-2)^2\,w'(s)^2-2\,w(s)^2+2\,w(s)^p$ is constant and therefore equal to $0$. This determines $w(0)=1$, so that the solution is unique and given by
\[
w(s)=(\cosh s)^{-\frac2{p-2}}\quad\forall\;s\in\R\;.
\]
Hence
\[
\mathsf K(\theta,p,\sigma)=\max_{\lambda>0}\frac{\(\lambda^{-1}\,I_p\)^\frac 2p}{\(\lambda\,J_2+ \sigma^2 \lambda^{-1}\,I_2\)^\theta\(\lambda^{-1}\,I_2\)^{1-\theta}}
\]
where
\[
I_q:=\int_{\R}|w(s)|^q\;ds\quad\mbox{and}\quad J_2:=\int_{\R}|w'(s)|^2\;ds\;.
\]
With $\lambda=\mu^\theta$, let $g(\mu):=\mu^{\frac 2p+2\theta-1}\,\frac{J_2}{I_2}+\sigma^2 \,\mu^{\frac 2p-1}$ and observe that
\[
\mathsf K(\theta,p,\sigma)=\frac{I_p^\frac 2p}{\big(\min_{\mu>0}g(\mu)\big)^\theta\,I_2}\;.
\]
If $\theta>\vartheta(p,1)$, the minimum of $g(\mu)$ is achieved at $\lambda^2=\mu^{2\theta}=\frac{p-2}{2+(2\theta-1)\,p}\,\sigma^2\,\frac{I_2}{J_2}$ and
\[
\big(\min_{\mu>0}g(\mu)\big)^\theta=h(p,\theta, \sigma)\,\(\frac{J_2}{I_2}\)^{\frac 12-\frac 1p}\quad\mbox{where}\quad h(p,\theta, \sigma):=\left[\frac{2+(2\theta-1)\,p}{\,(p-2)\,\sigma^2}\right]^\frac{p-2}{2\,p} \left[\frac{2\, p\,\theta\,\sigma^2}{2+(2\theta-1)\,p}\right]^\theta\,.
\]
If $\theta=\vartheta(p,1)$, then $\inf_{\mu>0}g(\mu)=\frac{J_2}{I_2}$ and we set $h(p,\theta,\sigma):=1$. For any $\theta \in \left[\vartheta(p,1), 1\right]$, we thus obtain
\[
\mathsf K(\theta,p,\sigma)=\frac{I_p^\frac 2p}{h(p,\theta,\sigma)\,J_2^{\frac 12-\frac 1p}\,I_2^{\frac 12+\frac 1p}}\;.
\]
Using the formula
\[
\int_{\R}\frac{ds}{(\cosh s)^q}=\frac{\sqrt\pi\;\Gamma\(\frac q2\)}{\Gamma\(\frac{q+1}2\)}=:f(q)\;,
\]
we can compute
\[
I_2=f\(\frac 4{p-2}\)\;,\quad I_p=f\(\frac{2\,p}{p-2}\)=f\(\frac 4{p-2}+2\)\;,
\]
and get
\[
I_2=\frac{\sqrt{\pi }\;\Gamma\left(\frac{2}{p-2}\right)}{\Gamma \left(\frac{p+2}{2\,(p-2)}\right)}\;,\quad I_p=\frac{4\,I_2}{p+2}\;,\quad J_2:=\frac 4{(p-2)^2}\(I_2-I_p\)=\frac{4\,I_2}{(p+2)(p-2)}\;.
\]
Hence
\[
\mathsf K(\theta,p, \sigma)=\left[\frac{(p-2)^2\,\sigma^2}{2+(2\theta-1)\,p}\right]^\frac{p-2}{2\,p}
\left[\frac{2+(2\theta-1)\,p}{2p\,\theta\,\sigma^2}\right]^\theta \left[\frac 4{p+2}\right]^\frac{6-p}{2\,p}\left[\frac{1}{I_2}\right]^\frac{p-2}{p}\;,
\]
which proves~\eqref{EqnK}.
\end{proof}

\subsection{The sharp Logarithmic Hardy inequality in the one-dimensional cylindric case}\label{Sec:LogHard1d}

With $2\,\gamma=p/(p-2)$ and $\theta=\gamma\,(p-2)$, we observe that the condition $\theta\in [\vartheta(p,1),1]$ is equivalent to $\gamma \in [1/(2\,p),1/(p-2)]$ and that $2+(2\theta-1)\,p=(2\,\gamma\,p-1)(p-2)$ is positive for any $p>2$ since $\gamma>1/(2\,p)$. Substituting $\theta$ with $\gamma\,(p-2)$ in the expression of $\mathsf K(\theta,p, \sigma)$ given by \eqref {EqnK} and taking the logarithm, we get
\begin{multline*}
\log\mathsf K(\gamma\,(p-2),p, \sigma)=\frac{p-2}{2\,p}\,\log\left[\frac{2 \sigma^2}{2\,\gamma\,p-1}\right] +\gamma\,(p-2)\,\log\left[\frac{2\,\gamma\,p-1}{2 \,p\,\gamma\,\sigma^2}\right]\\
+\frac{6-p}{2\,p}\,\log\left[\frac 4{p+2}\right] +\frac{p-2}{p}\,\log\left[\frac{\Gamma\left(\frac{2}{p-2}+\frac 12\right)}{\sqrt\pi\,\sqrt{\frac{2}{p-2}}\;\Gamma\left(\frac{2}{p-2}\right)}\right].
\end{multline*}
Using Stirling's formula, it is easy to see that $\lim_{t\to\infty}\frac{\Gamma\(t+\frac 12\)}{\sqrt t\,\Gamma(t)}=1$, so that $\lim_{p\to2}\mathsf K(\gamma\,(p-2),p, \sigma)=1$. Hence, for any $\gamma>1/(2\,p)$, let
\[
\mathcal K(\gamma, \sigma):=-2\,\frac d{dp}\Big[\mathsf K(\gamma\,(p-2),p, \sigma)\Big]_{|p=2}
\]
and consider the limit as $p\to 2$ in Inequality \eqref{Ineq:sigma}. We observe that $1/4>1/(2\,p)$ for any $p>2$ so that $\gamma>1/4$ guarantees $\gamma>1/(2\,p)$ uniformly in the limit $p\to 2_+$. The case $\gamma=1/4$ is achieved as a limit case.
\begin{lem}\label{Lem:LogHardy2} Let $\sigma>0$ and $\gamma\ge1/4$. Then for any $w\in H^1(\R)$ the following inequality holds true
\be{Eqn:LogSobR}
\ir{|w|^2\log\(\frac{|w|^2}{\ir{|w|^2}}\)}+\mathcal K(\gamma, \sigma)\ir{|w|^2}\leq2\, \gamma\ir{|w|^2}\;\log\left[\frac{\ir{|w'|^2}}{\ir{|w|^2}}+\sigma^2 \right]\,,
\ee
with sharp constant given by
\be{Kgamma}
\mathcal K(\gamma, \sigma)=2\,\gamma\,\log\gamma-\frac{4\,\gamma-1}2\,\log\(\frac{4\,\gamma-1}{4\,\sigma^2}\) +\frac 12\,\log\(2\,\pi\,e\)
\ee
if $\gamma>1/4$, and equality holds in \eqref{Eqn:LogSobR} for $w(s)=\exp\big(-\tfrac{\sigma^2}{4\,\gamma-1}s^2\big)$. If $\gamma=1/4$, then $\mathcal K(\gamma, \sigma)=2\,\gamma\,\log\gamma+\frac 12\,\log\(2\,\pi\,e\)$. 
\end{lem}
\begin{proof}The result follows by taking the logarithm of \eqref{Ineq:sigma} and differentiating at $p=2$ for $\gamma> 1/4$. The equality case in \eqref{Eqn:LogSobR} can be checked by a direct computation. \end{proof}

\subsection{The sharp inequalities for radial functions}\label{Sec:Radial-Interpolation}

Let $d\ge 2$ and consider the interpolation inequality~\eqref{Ineq:Gen_interp} restricted to the subset $\mathcal D_a^*(\R^d)$ of radial functions in $\mathcal D_{a}(\R^d)$, \emph{i.e.}
\[
\(\ird{\frac{|u|^p}{|x|^{b\,p}}}\)^\frac 2p\leq \mathsf C^*(\theta,p,a)\(\ird{\frac{|\nabla u|^2}{|x|^{2\,a}}}\)^\theta\(\ird{\frac{|u|^2}{|x|^{2\,(a+1)}}}\)^{1-\theta}\quad\forall\;u\in \mathcal D_a^*(\R^d)
\]
where $\mathsf C^*(\theta,p,a)$ denotes the best constant. Let $\sigma=(d-2-2\,a)/2$. By the Emden-Fowler change of coordinates \eqref{Transf:Emden-Fowler}, the above inequality is equivalent to
\be{Ineq:GN}
\(\icnd{|w|^p}\)^\frac 2p\leq \mathsf C^*(\theta,p,a)\(\icnd{| \nabla w|^2}+\sigma^2\icnd{|w|^2}\)^\theta\(\icnd{|w|^2}\)^{1-\theta}
\ee
for all functions $w\in H^1(\mathcal C)$ depending only on $s=-\log|x|$. Up to a normalization factor depending on $|\S|=2\,\pi^{d/2}/\,\Gamma\(d/2\)$, \eqref{Ineq:GN} is equivalent to the one-dimensional inequality \eqref{Ineq:sigma} with best constant $\mathsf K\(\theta,p, \sigma\)$. It is straightforward to check that $\mathsf C^*(\theta,p,a) = |\S|^{-(p-2)/p}\,\, \mathsf K\(\theta,p,\sigma\)$. We also note that the range of $\theta$ is as in Lemma~\ref{Lem:interpRadial}, that is, $\theta \in [\vartheta(p,1), \, 1]$.

\medskip Similarly, Inequality \eqref{Eqn:LogSobR} is equivalent to \eqref{Ineq:GLogHardyRad-w} with $\left[\mathcal K(\gamma, \sigma) + \log |\S| \right] = -2\,\gamma\,\log\mathsf C_{\rm GLH}^*$. This proves Theorem B''. Theorem B' follows by the Emden-Fowler change of coordinates \eqref{Transf:Emden-Fowler}. Theorem B corresponds to the special case $a=0$, $d\ge 3$. Notice that, with $\sigma=(d-2-2\,a)/2$, Inequality \eqref{Ineq:GLogHardyRad} written in terms of a function $f$ on $\R^+$ such that $u(x)=f(|x|)$ takes the form
\[
\int_0^\infty{r^{d-3-2\,a}\,|f|^2\log\(r^{d-2-2\,a}\,|f|^2\)\;dr}+\mathcal K\(\gamma,\sigma\) \leq 2\,\gamma\;\log\left[\int_0^\infty r^{d-1-2\,a}\,|f'|^2\;dr\right],
\]
under the normalization condition $\int_0^\infty r^{d-3-2\,a}\,|f|^2\,dr=1$.

\subsection{The sharp interpolation inequality in the case of the one-dimensional real line}\label{Sec:OneD}

Recall that inequalities written on the euclidean space $\R^d$ are equivalent to one-dimensional inequalities on $\mathcal C$ by the Emden-Fowler transformation~\eqref{Transf:Emden-Fowler} only in case of radial functions and, for $d=1$, only for \emph{even} functions. However, in this case, we may notice that the restriction $a<(d-2)/2=-1/2$ means that the weight $|x|^{-2a}$ corresponds to a positive power, so that we may consider the problem on $\R^+$ and $\R^-$ as two independent problems when dealing with a smooth function $u$ such that $u'(0)=0$.

Using $X_+\,\log X_++X_-\,\log X_-\le(X_++X_-)\,\log(X_++X_-)$ with $X_\pm=\int_{\R^\pm}\frac{|u|^2}{|x|^{2\,(a+1)}}\,dx$, we get
\begin{multline*}
\int_{\R}\frac{|u|^2}{|x|^{2\,(a+1)}}\,\log \(\frac{|x|^{d-2-2\,a}\,|u|^2}{\int_{\R}\frac{|u|^2}{|x|^{2\,(a+1)}}\,dx}\)\,dx=\int_{\R}\frac{|u|^2}{|x|^{2\,(a+1)}}\,\log \(|x|^{d-2-2\,a}\,|u|^2\)\,dx-(X_++X_-)\,\log(X_++X_-)\\
\le\int_{\R^-}\frac{|u|^2}{|x|^{2\,(a+1)}}\,\log \(|x|^{d-2-2\,a}\,|u|^2\)\,dx-X_-\,\log X_-+\int_{\R^+}\frac{|u|^2}{|x|^{2\,(a+1)}}\,\log \(|x|^{d-2-2\,a}\,|u|^2\)\,dx-X_+\,\log X_+\\
=\int_{\R^-}\frac{|u|^2}{|x|^{2\,(a+1)}}\,\log \(\frac{|x|^{d-2-2\,a}\,|u|^2}{\int_{\R^-}\frac{|u|^2}{|x|^{2\,(a+1)}}\,dx}\)\,dx+\int_{\R^+}\frac{|u|^2}{|x|^{2\,(a+1)}}\,\log \(\frac{|x|^{d-2-2\,a}\,|u|^2}{\int_{\R^+}\frac{|u|^2}{|x|^{2\,(a+1)}}\,dx}\)\,dx
\end{multline*}
By the Emden-Fowler transformation, we know that
\[
\int_{\R^\pm}\frac{|u|^2}{|x|^{2\,(a+1)}}\,\log \(\frac{|x|^{d-2-2\,a}\,|u|^2}{\int_{\R^\pm}\frac{|u|^2}{|x|^{2\,(a+1)}}\,dx}\)\,dx\le 2\,\gamma\,X_\pm\log\left[\mathsf C_{\rm GLH}\,\frac{Y_\pm}{X_\pm}\right]
\]
with $Y_\pm=\int_{\R^\pm}\frac{|\nabla u|^2}{|x|^{2\,a}}\,dx$. Using $X_-\log\left[\frac{Y_-}{X_-}\right]+X_+\log\left[\frac{Y_+}{X_+}\right]\le(X_++X_-)\log\left[\frac{Y_++Y_-}{X_++X_-}\right]$, we end up with the inequality
\[
\int_{\R}\frac{|u|^2}{|x|^{2\,(a+1)}}\,\log \(\frac{|x|^{d-2-2\,a}\,|u|^2}{\int_{\R}\frac{|u|^2}{|x|^{2\,(a+1)}}\,dx}\)\,dx\le 2\,\gamma\, {\int_{\R}\frac{|u|^2}{|x|^{2\,(a+1)}}\,dx}\,\log\left[\mathsf C_{\rm GLH}\,\frac{\int_{\R}\frac{|\nabla u|^2}{|x|^{2\,a}}\,dx} {\int_{\R}\frac{|u|^2}{|x|^{2\,(a+1)}}\,dx}\right]\,,
\]
which completes the proof of Theorem A' in the one-dimensional case.

\section{Connection with logarithmic Sobolev and Hardy inequalities}\label{Sec:LogRadial}

In this section, we study the connection of the \emph{Logarithmic Hardy inequality} \eqref{Ineq:LogHardy} and its generalized form \eqref{Ineq:HaGen} with the \emph{Euclidean Logarithmic Sobolev inequality} (see below), the \emph{Hardy inequality} \eqref{1} and its generalized form \eqref{Ineq:HS}, and the \emph{Logarithmic Sobolev inequality on $\mathcal C$} (see below).

As we have seen in the previous section, the weighted logarithmic Hardy inequality of Theorem B' (radial case) is equivalent to the one-dimensional inequality \eqref{Eqn:LogSobR} with sharp constant given by \eqref{Kgamma}. With the choice $2\sigma = 4\,\gamma-1$, we observe that $\lim_{\gamma\to 1/4}\mathcal K(\gamma,4\,\gamma-1)=\tfrac 12\,\log \(\frac{\pi\,e }{2}\)$ and recover the one-dimensional logarithmic Sobolev inequality written in the scale invariant form (see \cite{MR479373}) with optimal constant, namely
\be{Ineq:ls}
\ir{|w|^2\log\(\frac{|w|^2}{\ir{|w|^2}}\)}\leq\frac 12\ir{|w|^2}\;\log\left[\frac 2{\pi\,e}\,\frac{\ir{|w'|^2}}{\ir{|w|^2}}\right]\,.
\ee

Actually, Inequality~\eqref{Eqn:LogSobR} can be written in a simpler form in terms of a function $v\in H^1(\R)$, such that $w(s)=v\((d-2-2\,a)\,s/\sqrt{4\,\gamma-1}\)$ as follows. For any $\gamma\ge 1/4$ and any $v\in H^1(\R)$,
\[
\ir{|v|^2\log\(\frac{|v|^2}{\ir{|v|^2}}\)}+\left[2\,\gamma\,\log\gamma+\frac 12\,\log(2\,\pi\,e)\right]\ir{|v|^2}\leq2\,\gamma\!\ir{|v|^2}\;\log\left[\frac{\ir{|v'|^2}}{\ir{|v|^2}}+\gamma-\tfrac 14\right]\,.
\]
Consider the $\gamma$-dependent terms, \emph{i.e.}
\[
2\,\gamma\!\ir{|v|^2}\;\log\left[\frac{\ir{|v'|^2}}{\ir{|v|^2}}+\gamma-\tfrac 14\right]-2\,\gamma\,\log\gamma\,\ir{|v|^2}=\left[2\,\ir{|v'|^2}-\frac 12\ir{|v|^2}\right]\,f(t)\;,\]
with $f(t):=\tfrac 1t\,\log(t+1)$ and $t=\tfrac 1{4\,\gamma}\,\left[4\,\frac{\ir{|v'|^2}}{\ir{|v|^2}}-1\right]$. An elementary analysis shows that $f$ is decreasing so that, in terms of $\gamma$, the minimum of the right hand side is always achieved at $\gamma=1/4$. In this case we recover the one-dimensional logarithmic Sobolev inequality written in the scale invariant form~\eqref{Ineq:ls}. On the other hand, if we send $\gamma$ to $\infty$ which implies that $t \to 0$ and $f(t) \to 1$, we recover the logarithmic Sobolev inequality in the standard euclidean form (see~\cite{Gross75}):
\be {Ineq:ls-Gross}
\ir{|v|^2\log\(\frac{|v|^2}{\ir{|v|^2}}\)}+\frac 12\ir{|v|^2}\;\log\left[2\pi\,e^2\right]\le2\ir{|v'|^2}\,.
\ee

We can also recover Hardy's inequality from \eqref{Eqn:LogSobR} by taking the limit $\gamma\to+\infty$ and observing that $\lim_{\gamma\to+\infty}\mathcal K(\gamma,\sigma)/(2\,\gamma)=2\,\log\sigma$. The radial function $u\in\mathcal D^{1,2}_{a}(\R^d)$ given in terms of $w$ by the inverse of the Emden-Fowler change of coordinates \eqref{Transf:Emden-Fowler} satisfies
\be{Ineq:HardyGen}
\tfrac 14\,(d-2-2\,a)^2\ird{\frac{|u|^2}{|x|^{2\,(a+1)}}}\leq \ird{\frac{|\nabla u|^2}{|x|^{2\,a}}}\;.
\ee
This holds true for any $d\in\N$, $d\ge 2$ and any $a<0$. For $d\ge 3$, if we define $f(x):=|x|^{-a}\,u(x)$, $x\in\R^d$, then Inequality~\eqref{Ineq:HardyGen} is equivalent to the usual Hardy inequality (with $a=0$), namely
\be{Ineq:ha}
\tfrac 14\,(d-2)^2\ird{\frac{|f|^2}{|x|^2}}\leq \ird{|\nabla f|^2}\,, \quad \;f\in \mathcal D^{1,2}(\R^d)\;.
\ee
Using Schwarz' symmetrization, it is then straightforward to see that optimality is achieved for radial functions, thus showing that, with $\sigma=(d-2-2\,a)/2$, Inequality~\eqref{Eqn:LogSobR} implies Inequality~\eqref{Ineq:HardyGen} for any function $u\in\mathcal D^{1,2}_{a}(\R^d)$ (and not only for radial functions), that is Hardy's inequality if $a=0$ and all Caffarelli-Kohn-Nirenberg inequalities with $b=a+1$, $a<0$, otherwise.

Summarizing, the family of inequalities \eqref{Ineq:GLogHardyRad} of Theorem B' implies as extreme cases the logarithmic Sobolev inequality \eqref{Ineq:ls} at the endpoint $\gamma = 1/4$, and the euclidean logarithmic Sobolev inequality \eqref{Ineq:ls-Gross} and the Hardy inequality \eqref{Ineq:ha} as $\gamma$ tends to $+\infty$. In both cases, one-dimensional versions of the inequalities are involved. On $\mathcal C$, it is possible to recover the optimal logarithmic Sobolev inequality from the logarithmic Hardy inequality as follows.

\medskip Let $d\mu$ and and $d\nu_\sigma(t):=(2\pi\sigma^2)^{-1/2}\,\exp(-t^2/(2\sigma^2))\,dt$ be respectively the uniform probability measure on $\S$ induced by Lebesgue's measure on $\R^d$ and the gaussian probability measure on $\R$. Using the tensorization property of the logarithmic Sobolev inequalities (see for instance \cite{MR1845806}), we obtain the
\begin{lem}\label{Lem:LSIcndGen} For any $d\ge2$, the following inequality holds
\be{Ineq:LogSobCylinder}
\icnd{|w|^2\,\log\(\frac{|w|^2}{\icnd{|w|^2}}\)}+\mathcal K_d^\sigma\icnd{|w|^2}\le\max\left\{\frac 2{d-1},2\sigma^2\right\}\icnd{|\nabla w|^2}\quad\forall\;w\in H^1(\mathcal C)
\ee
with optimal constant
\[
\mathcal K_d^\sigma=\frac 12\,\log\(2\,\pi\,e^2\,\sigma^2\,|\S|^2\)=1+\frac 12\,\log\(\frac{8\,\pi^{d+1}\,\sigma^2}{\Gamma(d/2)^2}\)\;.
\]\end{lem}
\begin{proof} The sharp logarithmic Sobolev inequality on the sphere $\S$ is
\[
\int_{\S}|w|^2\,\log\(\frac{|w|^2}{\int_{\S}|w|^2\,d\mu}\)\;d\mu\le\frac 2{d-1}\int_{\S}|\nabla w|^2\;d\mu\quad\forall\;w\in H^1(\S)
\]
where $d\mu$ is the uniform probability measure on $\S$ induced by Lebesgue's measure in $\R^d$. It can be recovered as the limit as $q\to 2_+$ of sharp interpolation inequalities stated in \cite{MR1230930}, namely
\[
\frac 2{q-2}\left[\(\int_{\S}|w|^q\;d\mu\)^\frac 2q-\int_{\S}|w|^2\;d\mu\right]\le\frac 2{d-1}\int_{\S}|\nabla w|^2\;d\mu\;,
\]
and optimality is easily checked by considering the sequence of test functions $w_n=1+\varphi_1/n$, where $\varphi_1$ is a spherical harmonic function associated to the first non-zero eigenvalue of the Laplace-Beltrami operator on the sphere. On $\R$, the logarithmic Sobolev inequality associated to the gaussian probability measure $d\nu_\sigma$ has been established by L. Gross in \cite{Gross75}:
\[
\int_{\R}|w|^2\,\log\(\frac{|w|^2}{\int_{\R}|w|^2\,d\nu_\sigma}\)\;d\nu_\sigma\le 2\,\sigma^2\int_{\R}|w'|^2\;d\nu_\sigma\quad\forall\;w\in H^1(\R)\;.
\]

Again the constant $2\sigma^2$ is optimal as can be checked considering the sequence of test functions $w_n=1+\psi_1/n$, where $\psi_1(t)=t\,\exp(-t^2/(2\sigma)^2)$ is the first non constant Hermite function, up to a scaling.

The tensorization property of the logarithmic Sobolev inequalities shows that
\[
\int_{\mathcal C}|w|^2\,\log\(\frac{|w|^2}{\int_{\mathcal C}|w|^2\,d\mu\otimes d\nu_\sigma}\)\;d\mu\otimes d\nu_\sigma\le\max\left\{\frac 2{d-1},2\sigma^2\right\}\int_{\mathcal C}|\nabla w|^2\;d\mu\otimes d\nu_\sigma\;.
\]
Taking into account the normalization of $d\mu$ and $d\nu_\sigma$, $|\S|={2\,\pi^{d/2}}/{\Gamma\(d/2\)}$ and rewriting Gross' inequality with respect to Lebesgue's measure, we end up with \eqref{Ineq:LogSobCylinder}. The constant $\mathcal K_d^\sigma$ is optimal as can be shown again by considering a sequence of test functions based either on spherical harmonics or on Hermite functions.
\end{proof}

As a special case, for $\sigma^2=1/(d-1)$ and $\mathcal K_d:=\mathcal K_d^\sigma$, we have the following inequality on the cylinder.
\begin{cor}\label{Lem:LSIcnd} For any $d\ge2$, with $\mathcal K_d=1+\frac 12\,\log\big(\frac{8\,\pi^{d+1}}{(d-1)\,\Gamma(d/2)^2}\big)$, the following inequality holds
\be{Ineq:LSIcnd}
\icnd{|w|^2\,\log\(\frac{|w|^2}{\icnd{|w|^2}}\)}+\mathcal K_d\icnd{|w|^2}\le\frac 2{d-1}\icnd{|\nabla w|^2}\quad\forall\;w\in H^1(\mathcal C)\;.
\ee
\end{cor}

Using $\log(1+X)\le\alpha-1-\log\alpha+\alpha\,X$ for any $\alpha>0$, $X>0$, which amounts to write $\log Y\le Y-1$ with $Y=\alpha\,(X+1)$, and applying it in \eqref{Ineq:GLogHardy-w} with $X=\sigma^{-2}\icnd{|\nabla w|^2}/\icnd{|w|^2}$, $\sigma=(d-2-2\,a)/2$, $\alpha=\sigma^2/(\gamma\,(d-1))$, we deduce that
\[
\icnd{|w|^2\,\log\(\frac{|w|^2}{\icnd{|w|^2}}\)}- 2\,\gamma\left[\log\(\gamma\,(d-1)\,\mathsf C_{\rm GLH}\)+ \tfrac{\sigma^2}{\gamma\,(d-1)}-1\right]\icnd{|w|^2} \leq \frac 2{d-1}\,\icnd{|\nabla w|^2}
\]
for any $w\in H^1(\mathcal C)$. We may observe that
\[
\mathcal K_d+2\,\gamma\left[\log \(\gamma\,(d-1)\,\mathsf C_{\rm GLH}^*\) + \frac{\sigma^2}{\gamma\,(d-1)}-1\right]=\frac{4\,\gamma-1}2\left[Z-1-\log Z\right]\quad\mbox{with}\quad Z=\frac{4\,\sigma^2}{(4\,\gamma-1)\,(d-1)}\;.
\]
Hence, if $4\,\sigma^2=(4\,\gamma-1)\,(d-1)$ and if (for this specific value) $\mathsf C_{\rm GLH}=\mathsf C_{\rm GLH}^*$, then the optimal logarithmic Sobolev inequality \eqref{Ineq:LSIcnd} is a consequence of \eqref{Ineq:GLogHardy-w}.

We may observe that $4\,\sigma^2=(4\,\gamma-1)\,(d-1)$ means $Z=1$ and exactly corresponds to the threshold for the symmetry breaking result of Theorem~C. Notice that proving that $\mathsf C_{\rm GLH}=\mathsf C_{\rm GLH}^*$ is an open question.

\section{Symmetry breaking. Proof of Theorem C}\label{Sec:SymBreaking}

In this section we study the symmetry breaking of the Caffarelli-Kohn-Nirenberg interpolation inequality as well as of the logarithmic Hardy inequality. To achieve this, we use a technique introduced Catrina and Wang in \cite{Catrina-Wang-01} and later improved by Felli and Schneider in \cite{Felli-Schneider-03}. Also see \cite{MR1918754,MR2051129,DET}. The method amounts to consider a functional made of the difference of the two sides of the inequality, with a constant chosen so that the functional takes value zero in the optimal case, among radially symmetric functions, when the inequality is written for functions on $\R^d$. Equivalently, we can consider functions depending only on one real variable in the case of the cylinder. By linearizing around the optimal radial function, we obtain an explicit linear operator and can study when the eigenvalue corresponding to the subspace generated by the first non-trivial spherical harmonic function becomes negative. It is then clear that the functional can change sign, so that optimality cannot be achieved among radial functions. This proves the \emph{symmetry breaking}. We will apply the method first to the interpolation inequality \eqref{Ineq:Gen_interpw}, thus generalizing the results of Felli and Schneider to a more general Caffarelli-Kohn-Nirenberg interpolation inequality than the one they have considered, and then to the logarithmic Hardy inequality \eqref{Ineq:GLogHardy-w}.

\subsection{ Symmetry breaking for the interpolation inequality}\label{Sec:SymBreakinginterpolation}

Based on \eqref{Ineq:GN}, consider on $H^1(\mathcal C)$ the functional
\[
\mathcal J[w]:=\icnd{\(|\nabla w|^2+\tfrac 14\,(d-2-2\,a)^2\,|w|^2\)}-\left[\mathsf C^*(\theta,p,a)\right]^{-\frac 1\theta}\,\frac{\(\icnd{|w|^p}\)^\frac 2{p\,\theta}}{\(\icnd{|w|^2}\)^\frac{1-\theta}\theta}\;.
\]
Among functions $w\in H^1(\mathcal C)$ which depend only on $s$, $\mathcal J[w]$ is nonnegative, its minimum is zero and it is achieved by
\[
\overline w(y):=\big[\cosh(\lambda\,s)\big]^{-\frac 2{p-2}},\quad y=(s,\omega)\in\R \times\S=\mathcal C\;,\]
with $\lambda:=\tfrac 14\,(d-2-2\,a)\,(p-2)\,\sqrt{\tfrac{p+2}{2\,p\,\theta-(p-2)}}$. See the proof of Lemma~\ref{Lem:interpRadial} for more details. We can notice that
\[
\left[\mathsf C(\theta,p,a)\right]^{-\frac 1\theta}=\icnd{\(|\nabla \overline w|^2+\tfrac 14\,(d-2-2\,a)^2\,|\overline w|^2\)}\;\frac{\(\icnd{|\overline w|^2}\)^\frac{1-\theta}\theta}{\(\icnd{|\overline w|^p}\)^\frac 2{p\,\theta}}\;.
\]
With a slight abuse of notations, we shall write $\overline w$ as a function of $s$ only, which solves the ODE
\[
\lambda^2\,(p-2)^2\,w''-4\,w+2\,p\,|w|^{p-2}\,w=0
\]
and, as in the proof of Lemma~\ref{Lem:interpRadial},
\begin{eqnarray*}
&&\ir{|\overline w|^2}=\frac 1\lambda\,I_2\;,\\
&&\ir{|\overline w'|^2}=\lambda\,J_2=\frac{4\,\lambda}{p^2-4}\,I_2\;,\\
&&\ir{|\overline w|^p}=\frac 1\lambda\,I_p=\frac 4{p+2}\,\frac 1\lambda\,I_2\;.
\end{eqnarray*}
In terms of $p$ and $\theta$, we investigate the symmetry of optimal functions for \eqref{Ineq:Gen_interp} or, equivalently, for \eqref{Ineq:Gen_interpw} in the range
\[
0<p-2\le\frac 4{d-2}\quad\mbox{and}\quad \vartheta(p,d)\le \theta \le 1\;.
\]
Consider now $\mathcal J[\overline w+\varepsilon\,\phi]$ and Taylor expand it at order $2$ in $\varepsilon$, using the fact that $\overline w$ is a critical point and assuming that $\icnd{\overline w^{p-1}\,\phi}=0$:
\[
\frac 1{\varepsilon^2}\,\mathcal J[\overline w+\varepsilon\,\phi]=\icnd{|\nabla\phi|^2}-\kappa\icnd{\overline w^{p-2}\,|\phi|^2}+\mu\icnd{|\phi|^2}-\nu\(\icnd{\overline w\,\phi}\)^2+o(1)
\]
as $\varepsilon\to 0$, with
\begin{eqnarray*}
&&\kappa:=\frac{p-1}\theta\,\frac1{I_p}\,\(\lambda^2\,J_2+\tfrac 14\,(d-2-2\,a)^2\,I_2\)\;,\\
&&\mu:=\tfrac 14\,(d-2-2\,a)^2+\frac{1-\theta}\theta\,\frac 1{I_2}\,\(\lambda^2\,J_2+\tfrac 14\,(d-2-2\,a)^2\,I_2\)\,,\\
&&\nu:=\frac{1-\theta}{2\,\theta^2}\,\frac\lambda{I_2^2}\,\(\lambda^2\,J_2+\tfrac 14\,(d-2-2\,a)^2\,I_2\)\,.
\end{eqnarray*}
Spectral properties of the operator $\mathcal L:=-\Delta+\kappa\,\overline w^{p-2}+\mu$ are well known. Eigenfunctions can be characterized in terms of Legendre's polynomials, see for instance \cite[p. 74]{Landau-Lifschitz-67} and \cite{Felli-Schneider-03}. Using spherical coordinates and spherical harmonic functions, see \cite{MR0282313}, the discrete spectrum is made of the eigenvalues
\[
\lambda_{i,j}=\mu+i\,(d+i-2)-\frac{\lambda^2}4\(\sqrt{1+\frac{4\,\kappa}{\lambda^2}}-(1+2\,j)\)^2\quad \forall\;i\,,\;j\in\N\;,
\]
as long as $ \sqrt{1+4\,\kappa/\lambda^2}\ge 2\,j+1$. The eigenspace of $\mathcal L$ corresponding to $\lambda_{0,0}$ is generated by $\overline w$. Next we observe that the eigenfunction $\phi_{(1,0)}$ associated to $\lambda_{1,0}$ is not radially symmetric and such that $\icnd{\overline w\,\phi_{(1,0)}}=0$ and $\icnd{\overline w^{p-1}\,\phi_{(1,0)}}=0$. Hence, \emph{if $\lambda_{1,0}<0$, optimal functions for \eqref{Ineq:Gen_interpw} cannot be radially symmetric} and, as a consequence, $\mathsf C(\theta,p,a)>\mathsf C^*(\theta,p,a)$.

A lengthy computation allows to characterize for which values of $p$, $\theta$, $a$ and $d$, the eigenvalue $\lambda_{1,0}$ takes negative values. Using the fact that for any $a<(d-2)/2$ the quantity $2+p\,(2\,\theta-1)$ is positive, the corresponding condition turns out to be
\[
4\,p\,(d-1)\left(p^2+2\,p+8\,\theta -8\,\right)-\left(d^2+4\,a^2-4\,a\,(d-2)\right)(p-2)\,(p+2)^2 <0\;.
\]
This is never the case in the admissible range of our parameters if $0\le a<(d-2)/2$. On the other hand, for $a<0$ there is always a domain where symmetry breaking occurs. More precisely let
\[
\vartheta(p,d)=\frac{d\,(p-2)}{2\,p}\;,\quad\Theta(a,p,d):=\frac{p-2}{32\,(d-1)\,p}\, \left[ (p+2)^2\,(d^2 + 4\,a^2- 4\,a\,(d-2))-4\,p\,(p+4)\,(d-1) \right] \,.
\]
Symmetry breaking occurs if $\theta<\Theta(a,p,d)$. We observe that, for $p\in [2,2^*)$, we have $\vartheta(p,d)<\Theta(a,p,d)$ if and only if $a<a_-(p)$ with
\[
a_-(p):= \frac{d-2}{2}-\frac{2\,(d-1)}{p+2}\;.
\]
We note that the condition $\vartheta(p,d)<1$ is always satisfied under the assumption $p\in [2,2^*)$. On the other hand the condition $\Theta(a,p,d) \le 1$ is equivalent to
\[
a \ge \frac{d-2}{2}-\frac{2 \sqrt{d-1}}{\sqrt{(p-2)(p+2)}}\;.
\]
Finally, we notice that $d/4={\partial \vartheta(p,d)}/{\partial p}_{|p=2}<[1+{(d-2\,a-2)^2}/{(d-1)}]/4={\partial \Theta(a,p,d)}/{\partial p}_{|p=2}$ if $a<-1/2$. Summarizing these observations, we arrive at the following result.
\begin{thm}\label{Thm:SymmetryBreaking} Let $d\ge 2$, $2<p<2^*$ and $a<a_-(p)$. Optimality for \eqref{Ineq:Gen_interp} (resp. for \eqref{Ineq:Gen_interpw}) is not achieved among radial (resp. $s$-dependent) functions, that is, $\mathsf C(\theta, p,a)>\mathsf C^*(\theta, p,a)$ if either
\[
\vartheta(p,d)\le\theta<\Theta(a,p,d)\quad\mbox{when}\quad a \geq \frac{d-2}{2}-\frac{2 \sqrt{d-1}}{\sqrt{(p-2)(p+2)}}
\]
or
\[
\vartheta(p,d)\le\theta\leq 1\quad\mbox{when}\quad a<\frac{d-2}{2}-\frac{2 \sqrt{d-1}}{\sqrt{(p-2)(p+2)}}\;.
\]
In other words, symmetry breaking occurs for the optimal functions of \eqref{Ineq:Gen_interp} if $a$, $\theta$ and $p$ are in any of the two above regions. Moreover, if $a<-1/2$, there exists $\varepsilon >0$, $\gamma_1>d/4$ and $\gamma_2>\gamma_1$ such that symmetry breaking occurs if $\theta=\gamma\,(p-2)$ for any $\gamma\in(\gamma_1, \gamma_2)$ and any $p\in(2,2+\varepsilon)$.\end{thm}
An elementary computation shows that, $\Theta(a,p,d)>1$ amounts to \eqref{Cdt:FS}. This condition is the symmetry breaking condition for Caffarelli-Kohn-Nirenberg inequalities found in~\cite{Felli-Schneider-03}.

\subsection{Symmetry breaking for the weighted logarithmic Hardy inequality. Proof of Theorem~C}\label{Sec:SymBreakingLogHardy}

We now consider the weighted logarithmic Hardy inequalities of Theorem A' in the equivalent form of Theorem A''. As we have seen in Section \ref{Sec:interpol}, after the Emden-Fowler transformation these inequalities take the equivalent form:
\[
\icnd{|w|^2\log\(\frac{|w|^2}{\icnd{|w|^2}}\)}+2\,\gamma\,\log\mathsf C_{\rm GLH}^*\icnd{|w|^2}\leq2\,\gamma\!\icnd{|w|^2}\;\log\left[\frac{\icnd{|\nabla w|^2}}{\icnd{|w|^2}}+\sigma^2\right]
\]
with $\sigma=(d-2-2\,a)/2$. It is an open question to give sufficient conditions for which optimality is achieved among functions depending on $s$ only, so that $2\,\gamma\,\log\mathsf C_{\rm GLH}^*=\mathcal K(\gamma, \sigma)+\log|\S|$. We recall that equality among radial functions in \eqref{Kgamma} is achieved by
\[
\widetilde w(s,\omega):=|\S|^{-1/2}\,\overline w(s)\;,\quad y=\(s,\omega\)\in \R \times \S= \mathcal C\;,\]
where
\[
\overline w(s)=\(\frac{4\,\sigma^2}{2\,\pi\,(4\,\gamma-1)}\)^{1/4}\,\exp\(-\frac{\sigma^2\,s^2}{4\,\gamma-1}\)\quad\forall\;s\in\R\;.
\]
We note that $\widetilde w(s,\omega)$ is normalized to $1$ in $L^2(\mathcal C)$. As a consequence, it follows that
\[
\mathcal K(\gamma, \sigma)+\log|\S| = 2 \, \gamma \;\log \left[\icnd{|\nabla \widetilde{w}|^2}+\sigma^2 \right] - \icnd{|\widetilde{w}|^2\log |\widetilde{w}|^2}\,.
\]
After these preliminaries, consider the functional
\[
\mathcal F[w]:=\frac{\icnd{|\nabla w|^2}}{\icnd{|w|^2}}+\sigma^2- |\S|^\frac 1{2\,\gamma}\,\exp\left[\frac {\mathcal K(\gamma, \sigma)}{2\,\gamma}+\frac 1{2\,\gamma}\icnd{\frac{|w|^2}{\icnd{|w|^2}}\,\log\(\frac{|w|^2}{\icnd{|w|^2}}\)}\right]\,.
\]
We know that $\mathcal F[\widetilde w]=0$. Let
\[
\mathcal G[\phi]:=\lim_{\varepsilon\to 0}\frac{\mathcal F[\widetilde w+\varepsilon\,\phi]}{2\,\varepsilon^2}\;.
\]
We have in mind to consider an angle dependent perturbation of $\widetilde w$, so we shall assume that
\[
\icnd{\widetilde w\,\phi}=0\quad\mbox{and}\quad\icnd{\log|\widetilde w|^2\,\phi}=0\;.
\]
Under this assumption,
\begin{multline*}
\mathcal G[\phi]= \icnd{|\nabla\phi|^2}-\icnd{|\nabla\widetilde w|^2}\icnd{|\phi|^2}\\-\frac 1{2\,\gamma}\,\(\icnd{|\nabla\widetilde w|^2}+\sigma^2\)\cdot\left[\(2-\icnd{|\widetilde w|^2\,\log|\widetilde w|^2}\)\icnd{|\phi|^2}+\icnd{\log|\widetilde w|^2\,|\phi|^2}\right]\,.
\end{multline*}
After some elementary but tedious computations, one finds that
\[
\mathcal G[\phi]=\icnd{(\mathcal L\,\phi)\,\phi}\;,
\]
with $\mathcal L\,\phi:=-\Delta\,\phi+\frac 14\,A^2\,|s|^2\,\phi-\frac 32\,A\,\phi$ and $A:=\frac{4\,\sigma^2}{4\,\gamma-1}$. By separation of variables, it is straightforward to check that the spectrum of $\mathcal L$ is purely discrete and made of the eigenvalues
\[
\lambda_{i,j}=i\,(d+i-2)+A\,(j-1)\quad \forall\;i\,,\;j\in\N\;.
\]
It follows that $\lambda_{1,0}<0$ if
\be{ConCdt}
\frac{d}4 \le \gamma<\frac 14\,\(1+\frac{(d-2\,a-2)^2}{d-1}\)\,.
\ee
Hence symmetry breaking occurs provided that
\[
\frac{d}4 <\frac 14\,\(1+\frac{(d-2\,a-2)^2}{d-1}\)\;,
\]
which is equivalent to $a <- 1/2$. This concludes the proof of Theorem~C.

We recover the limit range for symmetry breaking in the interpolation inequalities studied in Section~\ref{Sec:SymBreakinginterpolation}. Condition~\eqref{ConCdt} asymptotically defines a cone in which symmetry breaking occurs (see Theorem~\ref{Thm:SymmetryBreaking}) given by $d/4={\partial \vartheta(p,d)}/{\partial p}_{|p=2}<\gamma<[1+{(d-2\,a-2)^2}/{(d-1)}]/4={\partial \Theta(a,p,d)}/{\partial p}_{|p=2}$.\clearpage

\subsection{Plots}\label{Sec:Plots}$\phantom{x}$

\begin{figure}[!h]\includegraphics[width=3.5cm]{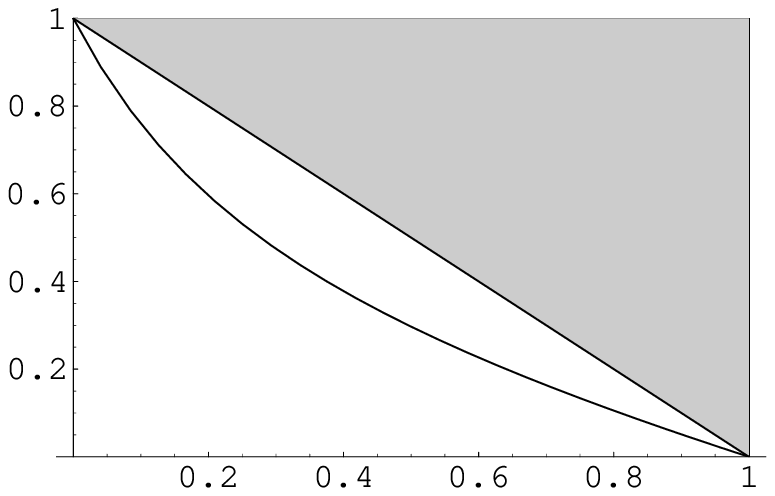}\hspace*{12pt}\includegraphics[width=3.5cm]{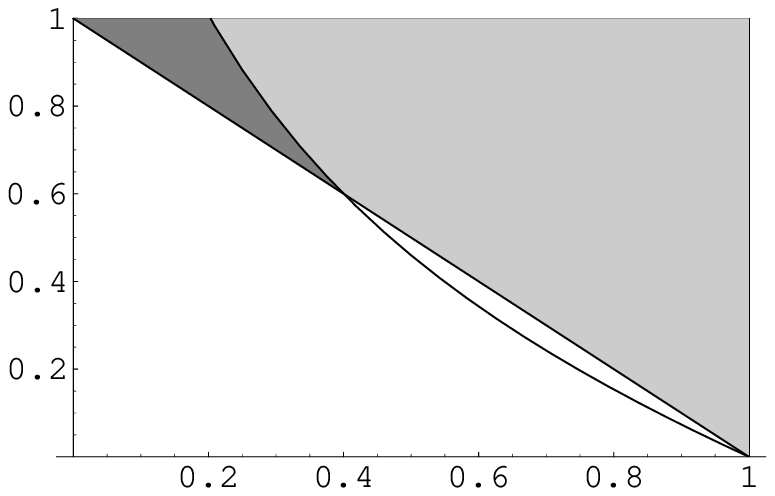}\hspace*{12pt}\includegraphics[width=3.5cm]{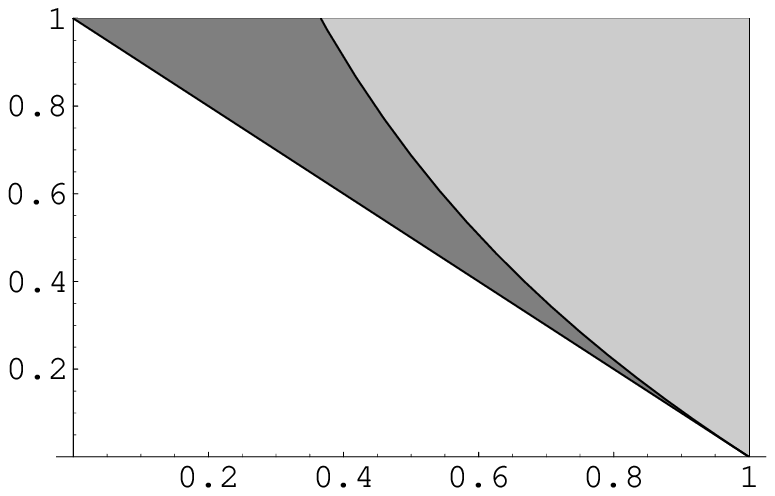}\hspace*{12pt}\includegraphics[width=3.5cm]{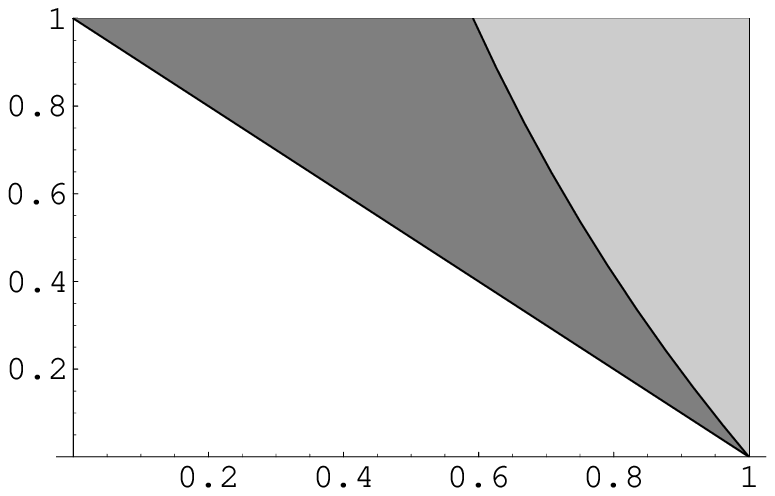}\caption{\small Plot of the admissible regions (gray areas) with symmetry breaking region established in Theorem~\ref{Thm:SymmetryBreaking} (dark grey) in $(\eta,\theta)$ coordinates, with $\eta:=b-a$, for various values of $a$, in dimension $d=3$: from left to right, $a=0$, $a=-0.25$, $a=-0.5$ and $a=-1$. The two curves are $\eta\mapsto\vartheta(p,d)=1-\eta$ and $\eta\mapsto\Theta(a,p,d)$, for $p=2\,d/(d-2+2\,\eta)$. In the range $a\in(-1/2,0)$, they intersect for $a=a_-(p)$, \emph{i.e.} $\eta=2\,a\,(1-d)/(d+2\,a)$. They are tangent at $(\eta,\theta)=(1,0)$ for $a=-1/2$. The symmetry breaking region contains a cone attached to $(\eta,\theta)=(1,0)$ for $a<-1/2$, which determines values of $\gamma$ for which symmetry breaking occurs in the logarithmic Hardy inequality.}\end{figure}

\begin{figure}[!h]\includegraphics[width=5cm]{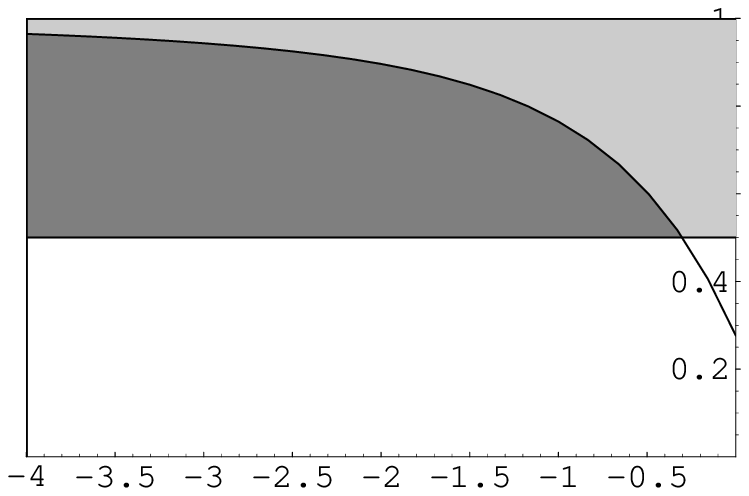}\hspace*{2cm}\includegraphics[width=5cm]{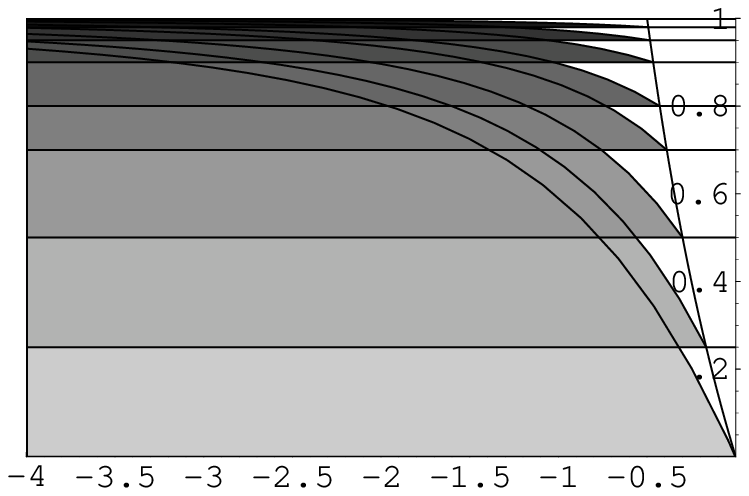}\caption{\small\emph{Left.--} For a given value of $\theta\in (0,1]$, admissible values of the parameters for which \eqref{Ineq:Gen_interp} holds are given by \eqref{Range}: in terms of $(a,\eta)$ with $\eta=b-a$, this simply means $\eta\ge 1-\theta$ (grey areas). According to Theorem~\ref{Thm:SymmetryBreaking}, symmetry breaking occurs if $\theta<\Theta(a,p,d)$, which determines a region $\eta<g(a,\theta)$ (dark grey). Notice that $\eta<g(a,1)$ corresponds to Condition \eqref{Cdt:FS} found by Felli and Schneider. The plot corresponds to $d=3$ and $\theta=0.5$. \emph{Right.--} Regions of symmetry breaking, \emph{i.e.} $1-\theta\le\eta<g(a,\theta)$, are shown for $\theta=1$, $0.75$, $0.5$, $0.3$, $0.2$, $0.1$, $0.05$, $0.02$. For each value of $\theta$, the supremum value for which symmetry breaking has been established is $a=a_-(p)$ for $p=2\,d/(d-2\,\theta)$, which determines a curve $\eta=h(a)$ by requiring that $\theta=1-\eta$. The limit case $\eta=0=h(0)$ corresponds to the case studied by Felli and Schneider, while $h(-1/2)=1$ determines the supremum value for which symmetry breaking has been established in the limit case $\eta=1$, \emph{i.e.} $p=2$, consistently with Theorem C.}\end{figure}

\section*{Concluding remarks}

The purpose of this paper is to establish a new family of inequalities in the Euclidean space $\R^d$ and in the cylinder $\R\times\S$. These inequalities are stronger than Hardy's inequality and the logarithmic Sobolev inequalities, but are related to both of them and, for this reason, we have called them \emph{logarithmic Hardy inequalities}. They are invariant term by term under scaling, which distinguishes them from usual logarithmic Sobolev inequalities. On $\R^d$, they are written for unbounded measures and, as far as we know, cannot be easily reduced to inequalities written for probability measures or for Lebesgue's measure. They also appear as an endpoint of a family of Caffarelli-Kohn-Nirenberg inequalities, which is more general than the subfamily studied for instance by Catrina and Wang in \cite{Catrina-Wang-01}.

A very natural question is to determine whether optimal functions on $\R^d$ are radially symmetric or not. Using the method introduced by Catrina and Wang in \cite{Catrina-Wang-01} and extended in \cite{Felli-Schneider-03} by Felli and Schneider, we prove that optimal functions in $\R^d$ are not radially symmetric functions in the case of the general Caffarelli-Kohn-Nirenberg inequalities and in the corresponding logarithmic Hardy inequalities, for parameters in a certain range. The method is rather simple. It amounts to linearize the inequality around an optimal function among radial functions and study the sign of the first eigenvalue of an associated operator. A negative eigenvalue then means that optimality is achieved among non-radial functions. The results of \emph{symmetry breaking} that we obtain by this method are fully consistent with previously known results. They allow us to characterize a whole region where the weights are strong enough to break the symmetry that would naturally arise from the nonlinearity in the absence of weights (and can then be proved either by symmetrization techniques or by moving plane methods as in \cite{MR634248}). The symmetry region is by far less understood, although it has recently been established in \cite{0902} that both regions are separated by a curve (in the case of the subfamily considered by Catrina and Wang). In the general form of the Caffarelli-Kohn-Nirenberg inequalities, there is an additional term which competes with the nonlinearity to break the symmetry, thus making the analysis more difficult. Hence the main challenge is now to establish the range for symmetry of the optimal functions. This would have some interesting consequences. As mentioned in Section~\ref {Sec:LogRadial}, if, for instance, symmetry holds in the complementary region of the one for which symmetry breaking has been established, then we would recover the optimal logarithmic Sobolev inequality on the cylinder as a direct consequence of the logarithmic Hardy inequality.

\bigskip\noindent{\bf Acknowledgments.} J.D. has been supported by the ECOS contract no. C05E09 and the ANR grants \emph{IFO}, \emph{EVOL} and \emph{CBDif}, and thanks the department of Mathematics of the University of Crete and the Departamento de Ingenier\'{\i}a Matem\'atica of the University of Chile for their warm hospitality. This work is also part of the MathAmSud \emph{NAPDE} project (M.d.P. \& J.D.). Plots have been done with Mathematica$^{\mbox{\tiny{\sc tm}}}$.\par\smallskip\noindent{\sl\small \copyright~2009 by the authors. This paper may be reproduced, in its entirety, for noncommercial purposes.}


\small \begin{thebibliography}{10}

\bibitem{MR2142067}
{\sc B.~Abdellaoui, E.~Colorado, and I.~Peral}, {\em Some improved
  {C}affarelli-{K}ohn-{N}irenberg inequalities}, Calc. Var. Partial
  Differential Equations, 23 (2005), pp.~327--345.

\bibitem{MR1862130}
{\sc Adimurthi, N.~Chaudhuri, and M.~Ramaswamy}, {\em An improved
  {H}ardy-{S}obolev inequality and its application}, Proc. Amer. Math. Soc.,
  130 (2002), pp.~489--505.

\bibitem{Adi-Filippas-Tertikas}
{\sc Adimurthi, S.~Filippas, and A.~Tertikas}, {\em On the best constant of
  {H}ardy-{S}obolev inequalities}, Nonlinear Anal., 70 (2009), pp.~2826--2833.

\bibitem{MR2427077}
{\sc M.~Agueh}, {\em Gagliardo-{N}irenberg inequalities involving the gradient
  {$L\sp 2$}-norm}, C. R. Math. Acad. Sci. Paris, 346 (2008), pp.~757--762.

\bibitem{MR2198838}
{\sc A.~Alvino, V.~Ferone, and G.~Trombetti}, {\em On the best constant in a
  {H}ardy-{S}obolev inequality}, Appl. Anal., 85 (2006), pp.~171--180.

\bibitem{MR1845806}
{\sc C.~An{\'e}, S.~Blach{\`e}re, D.~Chafa{\"{\i}}, P.~Foug{\`e}res, I.~Gentil,
  F.~Malrieu, C.~Roberto, and G.~Scheffer}, {\em Sur les in\'egalit\'es de
  {S}obolev logarithmiques}, vol.~10 of Panoramas et Synth\`eses [Panoramas and
  Syntheses], Soci\'et\'e Math\'ematique de France, Paris, 2000.
\newblock With a preface by D. Bakry and M. Ledoux.

\bibitem{0528}
{\sc A.~Arnold, J.-P. Bartier, and J.~Dolbeault}, {\em Interpolation between
  logarithmic {S}obolev and {P}oincar{{\'e}} inequalities}, Communications in
  Mathematical Sciences, 5 (2007), pp.~971--979.

\bibitem{MR1842428}
{\sc A.~Arnold, P.~Markowich, G.~Toscani, and A.~Unterreiter}, {\em On convex
  {S}obolev inequalities and the rate of convergence to equilibrium for
  {F}okker-{P}lanck type equations}, Comm. Partial Differential Equations, 26
  (2001), pp.~43--100.

\bibitem{MR0448404}
{\sc T.~Aubin}, {\em Probl\`emes isop\'erim\'etriques et espaces de {S}obolev},
  J. Differential Geometry, 11 (1976), pp.~573--598.

\bibitem{MR2354734}
{\sc F.~G. Avkhadiev and K.-J. Wirths}, {\em Unified {P}oincar\'e and {H}ardy
  inequalities with sharp constants for convex domains}, ZAMM Z. Angew. Math.
  Mech., 87 (2007), pp.~632--642.

\bibitem{MR1918928}
{\sc M.~Badiale and G.~Tarantello}, {\em A {S}obolev-{H}ardy inequality with
  applications to a nonlinear elliptic equation arising in astrophysics}, Arch.
  Ration. Mech. Anal., 163 (2002), pp.~259--293.

\bibitem{MR772092}
{\sc D.~Bakry and M.~{\'E}mery}, {\em Hypercontractivit\'e de semi-groupes de
  diffusion}, C. R. Acad. Sci. Paris S\'er. I Math., 299 (1984), pp.~775--778.

\bibitem{ca-ba-ro2}
{\sc F.~Barthe, P.~Cattiaux, and C.~Roberto}, {\em Concentration for
  independent random variables with heavy tails}, AMRX Appl. Math. Res.
  Express, 2 (2005), pp.~39--60.

\bibitem{bartier2009improved}
{\sc J.~Bartier, A.~Blanchet, J.~Dolbeault, and M.~Escobedo}, {\em {Improved
  intermediate asymptotics for the heat equation}}, Arxiv preprint 0908.2226,
  (2009).

\bibitem{MR2201954}
{\sc J.-P. Bartier and J.~Dolbeault}, {\em Convex {S}obolev inequalities and
  spectral gap}, C. R. Math. Acad. Sci. Paris, 342 (2006), pp.~307--312.

\bibitem{MR1230930}
{\sc W.~Beckner}, {\em Sharp {S}obolev inequalities on the sphere and the
  {M}oser-{T}rudinger inequality}, Ann. of Math. (2), 138 (1993), pp.~213--242.

\bibitem{MR2424899}
{\sc R.~D. Benguria, R.~L. Frank, and M.~Loss}, {\em The sharp constant in the
  {H}ardy-{S}obolev-{M}az'ya inequality in the three dimensional upper
  half-space}, Math. Res. Lett., 15 (2008), pp.~613--622.

\bibitem{Benguria-Loss03}
{\sc R.~D. Benguria and M.~Loss}, {\em Connection between the {L}ieb-{T}hirring
  conjecture for {S}chr\"odinger operators and an isoperimetric problem for
  ovals on the plane}, in Partial differential equations and inverse problems,
  vol.~362 of Contemp. Math., Amer. Math. Soc., Providence, RI, 2004,
  pp.~53--61.

\bibitem{MR0282313}
{\sc M.~Berger, P.~Gauduchon, and E.~Mazet}, {\em Le spectre d'une vari\'et\'e
  riemannienne}, Lecture Notes in Mathematics, Vol. 194, Springer-Verlag,
  Berlin, 1971.

\bibitem{Blanchet:2009sf}
{\sc A.~Blanchet, M.~Bonforte, J.~Dolbeault, G.~Grillo, and J.~V{\'a}zquez},
  {\em Asymptotics of the fast diffusion equation via entropy estimates},
  Archive for Rational Mechanics and Analysis, 191 (2009), pp.~347--385.

\bibitem{MR1682772}
{\sc S.~G. Bobkov and F.~G{\"o}tze}, {\em Exponential integrability and
  transportation cost related to logarithmic {S}obolev inequalities}, J. Funct.
  Anal., 163 (1999), pp.~1--28.

\bibitem{bonforte-2009}
{\sc M.~Bonforte, J.~Dolbeault, G.~Grillo, and J.-L. Vazquez}, {\em Sharp rates
  of decay of solutions to the nonlinear fast diffusion equation via functional
  inequalities}, Arxiv preprint 0907.2986,  (2009).

\bibitem{BrV}
{\sc H.~Brezis and J.~L. V{\'a}zquez}, {\em Blow-up solutions of some nonlinear
  elliptic problems}, Rev. Mat. Univ. Complut. Madrid, 10 (1997), pp.~443--469.

\bibitem{MR1918754}
{\sc J.~Byeon and Z.-Q. Wang}, {\em Symmetry breaking of extremal functions for
  the {C}affarelli-{K}ohn-{N}irenberg inequalities}, Commun. Contemp. Math., 4
  (2002), pp.~457--465.

\bibitem{Caffarelli-Kohn-Nirenberg-84}
{\sc L.~Caffarelli, R.~Kohn, and L.~Nirenberg}, {\em First order interpolation
  inequalities with weights}, Compositio Math., 53 (1984), pp.~259--275.

\bibitem{MR2079069}
{\sc E.~Carlen and M.~Loss}, {\em Logarithmic {S}obolev inequalities and
  spectral gaps}, in Recent advances in the theory and applications of mass
  transport, vol.~353 of Contemp. Math., Amer. Math. Soc., Providence, RI,
  2004, pp.~53--60.

\bibitem{Catrina-Wang-01}
{\sc F.~Catrina and Z.-Q. Wang}, {\em On the {C}affarelli-{K}ohn-{N}irenberg
  inequalities: sharp constants, existence (and nonexistence), and symmetry of
  extremal functions}, Comm. Pure Appl. Math., 54 (2001), pp.~229--258.

\bibitem{MR1223899}
{\sc K.~S. Chou and C.~W. Chu}, {\em On the best constant for a weighted
  {S}obolev-{H}ardy inequality}, J. London Math. Soc. (2), 48 (1993),
  pp.~137--151.

\bibitem{Cianchi-Ferone}
{\sc A.~Cianchi and A.~Ferone}, {\em Hardy inequalities with non-standard
  remainder terms}, Ann. Inst. H. Poincar\'e Anal. Non Lin\'eaire, 25 (2008),
  pp.~889--906.

\bibitem{MR1103113}
{\sc E.~B. Davies}, {\em Heat kernels and spectral theory}, vol.~92 of
  Cambridge Tracts in Mathematics, Cambridge University Press, Cambridge, 1990.

\bibitem{MR2060479}
{\sc J.~D{\'a}vila and L.~Dupaigne}, {\em Hardy-type inequalities}, J. Eur.
  Math. Soc. (JEMS), 6 (2004), pp.~335--365.

\bibitem{MR1940370}
{\sc M.~Del~Pino and J.~Dolbeault}, {\em Best constants for
  {G}agliardo-{N}irenberg inequalities and applications to nonlinear
  diffusions}, J. Math. Pures Appl. (9), 81 (2002), pp.~847--875.

\bibitem{MR1957678}
\leavevmode\vrule height 2pt depth -1.6pt width 23pt, {\em The optimal
  {E}uclidean {$L\sp p$}-{S}obolev logarithmic inequality}, J. Funct. Anal.,
  197 (2003), pp.~151--161.

\bibitem{0902}
{\sc J.~Dolbeault, M.~J. Esteban, M.~Loss, and G.~Tarantello}, {\em On the
  symmetry of extremals for the {C}affarelli-{K}ohn-{N}irenberg inequalities},
  Advanced Nonlinear Studies, 9 (2009), pp.~713--727.

\bibitem{DET}
{\sc J.~Dolbeault, M.~J. Esteban, and G.~Tarantello}, {\em The role of {O}nofri
  type inequalities in the symmetry properties of extremals for
  {C}affarelli-{K}ohn-{N}irenberg inequalities, in two space dimensions}, Ann.
  Sc. Norm. Super. Pisa Cl. Sci. (5), 7 (2008), pp.~313--341.

\bibitem{MR2448650}
{\sc J.~Dolbeault, B.~Nazaret, and G.~Savar{\'e}}, {\em A new class of
  transport distances between measures}, Calc. Var. Partial Differential
  Equations, 34 (2009), pp.~193--231.

\bibitem{Felli-Schneider-03}
{\sc V.~Felli and M.~Schneider}, {\em Perturbation results of critical elliptic
  equations of {C}affarelli-{K}ohn-{N}irenberg type}, J. Differential
  Equations, 191 (2003), pp.~121--142.

\bibitem{MR2214621}
{\sc S.~Filippas, V.~Maz{\cprime}ya, and A.~Tertikas}, {\em On a question of
  {B}rezis and {M}arcus}, Calc. Var. Partial Differential Equations, 25 (2006),
  pp.~491--501.

\bibitem{MR2308757}
{\sc S.~Filippas, L.~Moschini, and A.~Tertikas}, {\em Sharp two-sided heat
  kernel estimates for critical {S}chr\"odinger operators on bounded domains},
  Comm. Math. Phys., 273 (2007), pp.~237--281.

\bibitem{FMT}
\leavevmode\vrule height 2pt depth -1.6pt width 23pt, {\em {Improving $L^2$
  estimates to Harnack inequalities}}, Proc. London Math. Soc. (3), 99 (2009),
  pp.~326--352.

\bibitem{MR1918494}
{\sc S.~Filippas and A.~Tertikas}, {\em Optimizing improved {H}ardy
  inequalities}, J. Funct. Anal., 192 (2002), pp.~186--233.

\bibitem{MR2462588}
\leavevmode\vrule height 2pt depth -1.6pt width 23pt, {\em Corrigendum to:
  \cite{MR1918494}}, J. Funct. Anal., 255 (2008), p.~2095.

\bibitem{MR634248}
{\sc B.~Gidas, W.~M. Ni, and L.~Nirenberg}, {\em Symmetry of positive solutions
  of nonlinear elliptic equations in {${\R}\sp{n}$}}, in Mathematical analysis
  and applications, {P}art {A}, vol.~7 of Adv. in Math. Suppl. Stud., Academic
  Press, New York, 1981, pp.~369--402.

\bibitem{Gross75}
{\sc L.~Gross}, {\em Logarithmic {S}obolev inequalities}, Amer. J. Math., 97
  (1975), pp.~1061--1083.

\bibitem{MR1892180}
{\sc M.~Hoffmann-Ostenhof, T.~Hoffmann-Ostenhof, and A.~Laptev}, {\em A
  geometrical version of {H}ardy's inequality}, J. Funct. Anal., 189 (2002),
  pp.~539--548.

\bibitem{MR1731336}
{\sc T.~Horiuchi}, {\em Best constant in weighted {S}obolev inequality with
  weights being powers of distance from the origin}, J. Inequal. Appl., 1
  (1997), pp.~275--292.

\bibitem{Landau-Lifschitz-67}
{\sc L.~Landau and E.~Lifschitz}, {\em Physique th\'eorique. Tome III:
  M\'ecanique quantique. Th\'eorie non relativiste. (French)}, Deuxi\`eme
  \'edition. Translated from russian by E. Gloukhian. \'Editions Mir, Moscow,
  1967.

\bibitem{MR2051129}
{\sc C.-S. Lin and Z.-Q. Wang}, {\em Symmetry of extremal functions for the
  {C}affarelli-{K}ohn-{N}irenberg inequalities}, Proc. Amer. Math. Soc., 132
  (2004), pp.~1685--1691.

\bibitem{MR0311856}
{\sc B.~Muckenhoupt}, {\em Hardy's inequality with weights}, Studia Math., 44
  (1972), pp.~31--38.
\newblock Collection of articles honoring the completion by Antoni Zygmund of
  50 years of scientific activity,~ I.

\bibitem{MR2317190}
{\sc J.~H. Petersson}, {\em Best constants for {G}agliardo-{N}irenberg
  inequalities on the real line}, Nonlinear Anal., 67 (2007), pp.~587--600.

\bibitem{MR0463908}
{\sc G.~Talenti}, {\em Best constant in {S}obolev inequality}, Ann. Mat. Pura
  Appl. (4), 110 (1976), pp.~353--372.

\bibitem{Tertikas-Tintarev}
{\sc A.~Tertikas and K.~Tintarev}, {\em On existence of minimizers for the
  {H}ardy-{S}obolev-{M}az'ya inequality}, Ann. Mat. Pura Appl.,  (2007).

\bibitem{MR1447044}
{\sc G.~Toscani}, {\em Sur l'in\'egalit\'e logarithmique de {S}obolev}, C. R.
  Acad. Sci. Paris S\'er. I Math., 324 (1997), pp.~689--694.

\bibitem{MR1760280}
{\sc J.~L. V{\'a}zquez and E.~Zuazua}, {\em The {H}ardy inequality and the
  asymptotic behaviour of the heat equation with an inverse-square potential},
  J. Funct. Anal., 173 (2000), pp.~103--153.

\bibitem{MR1923362}
{\sc E.~J.~M. Veling}, {\em Lower bounds for the infimum of the spectrum of the
  {S}chr\"odinger operator in {$\mathbb R\sp N$} and the {S}obolev
  inequalities}, JIPAM. J. Inequal. Pure Appl. Math., 3 (2002), pp.~Article 63,
  22 pp..

\bibitem{MR2048612}
\leavevmode\vrule height 2pt depth -1.6pt width 23pt, {\em Corrigendum on the
  paper: ``{L}ower bounds for the infimum of the spectrum of the
  {S}chr\"odinger operator in {${\mathbb R}\sp N$} and the {S}obolev
  inequalities'' [{JIPAM}. {J}. {I}nequal. {P}ure {A}ppl. {M}ath. {\bf 3}
  (2002), no. 4, {A}rticle 63, 22 pp; mr1923362]}, JIPAM. J. Inequal. Pure
  Appl. Math., 4 (2003), pp.~Article 109, 2 pp..

\bibitem{MR2003359}
{\sc Z.-Q. Wang and M.~Willem}, {\em Caffarelli-{K}ohn-{N}irenberg inequalities
  with remainder terms}, J. Funct. Anal., 203 (2003), pp.~550--568.

\bibitem{MR479373}
{\sc F.~B. Weissler}, {\em Logarithmic {S}obolev inequalities for the
  heat-diffusion semigroup}, Trans. Amer. Math. Soc., 237 (1978), pp.~255--269.

\end{thebibliography}

\end{document}